\documentclass[12pt]{article}

\providecommand{\MR}{\relax\ifhmode\unskip\space\fi MR }

\providecommand{\href}[2]{#2}

\usepackage[top=30truemm,bottom=28truemm,left=28truemm,right=28truemm]{geometry}

\usepackage[affil-it]{authblk}
\makeatletter
\def\@maketitle{%
  \newpage
  \null
  \vskip 2em%
  \begin{center}%
  \let \footnote \thanks
    {\Large\bfseries \@title \par}%
    \vskip 1.5em%
    {\normalsize
      \lineskip .5em%
      \begin{tabular}[t]{c}%
        \@author
      \end{tabular}\par}%
    \vskip 1em%
    {\normalsize \@date}%
  \end{center}%
  \par
  \vskip 1.5em}
\makeatother
\usepackage{amsmath}
\usepackage{mathrsfs}
\usepackage{amssymb}
\usepackage{amsfonts}
\usepackage{amsthm}
\usepackage{amscd}
\usepackage{graphicx}
\usepackage{bbold}

\usepackage{tikz}
\usetikzlibrary{matrix,arrows,decorations.pathmorphing}
\usepackage{tikz-cd}
\usetikzlibrary{arrows}
\tikzset{commutative diagrams/.cd,arrow style=tikz,diagrams={>=latex'}}

\mathchardef\-="2D 

\newcommand{\N}{\mathbb{N}}

\numberwithin{equation}{section}
\numberwithin{figure}{section}
\theoremstyle{plain}
\newtheorem{theorem}{Theorem}[section]

\theoremstyle{definition}

\newtheorem{lemma}[theorem]{Lemma}
\newtheorem{corollary}[theorem]{Corollary}

\newtheorem*{fact*}{}

\begin{document}

\setcounter{tocdepth}{6} 
\setcounter{secnumdepth}{6}

\title{On the finiteness of Carmichael numbers with Fermat factors and $L=2^{\alpha}P^2$.}
\author{Yu Tsumura}
\affil{Department of Mathematics\\ The Ohio State University\\ Columbus, OH 43210, USA}
\affil{\textit {tsumura.2@osu.edu}}


\maketitle

\begin{abstract}
Let $m$ be a Carmichael number and let $L$ be the least common multiple of $p-1$, where $p$ runs over the prime factors of $m$.
We determine all the Carmichael numbers $m$ with a Fermat prime factor such that $L=2^{\alpha}P^2$, where $k\in \N$ and $P$ is an odd prime number. There are eleven such Carmichael numbers.
\end{abstract}

\section{Introduction}
Fermat's little theorem states that a prime number $p$ divides $a^p-a$ for any $a\in \N$.
It would be interesting if prime numbers are the only integers (except for $1$) having this property.
Some composites $m$, however, satisfies $a^m\equiv a \pmod{m}$ for all integers $a$. Such a composite is called a \textit{Carmichael number}.
The smallest Carmichael number is $561$ and it was found by Carmichael in 1910 \cite{MR1558896}.
Prior to Carmichael's discovery, Korselt \cite{Korselt1899} had provided the following simple test for Carmichael numbers in 1899:

\begin{fact*}[Korselt's criterion] A composite $m$ is a Carmichael number if and only if $m$ is squarefree and $p-1$ divides $m-1$ for all prime divisors $p$ of $m$.
\end{fact*}

Let $L=L_m$ be the least common multiple of $p_i-1$, where $m=p_1\cdots p_n$ is the prime decomposition of a squarefree composite.
Then Korselt's criterion can be rephrased as follows: $m$ is a Carmichael number if and only if $L\mid m-1$.

It is known that there are infinitely many Carmichael numbers. The result was proved by Alfred, Granville, and Pomerance in 1994 \cite{MR1283874} based on  Erd\H{o}s' heuristic argument \cite{MR0079031}.
The idea is to construct an integer $L
'$ so that it is divisible by $p-1$ for a large number of primes $p$.
Then if the product $m=p_1\cdots p_k$ of some of these primes is congruent to $1$ modulo $L'$, then $m$ is a Carmichael number by Korselt's criterion since we have $L_m \mid L' \mid m-1$.
\cite{MR1283874} showed that in fact there are infinitely many such $L'$ and  such products, hence the infinitude of Carmichael numbers follows.
Wright noted that as $L'$ contains a sizable number of prime factors, 
it is likely that $L_m$ corresponding to the Carmichael numbers $m$ obtained in this manner contains many prime factors as well.
This led him the study of Carmichael numbers with restricted $L$.
Wright \cite{MR2988558} proved that there are no Carmichael numbers with $L=2^{\alpha}$ and determined all the Carmichael numbers with $L=2^{\alpha}P$ for some odd prime $P$ under the assumption that the Fermat primes conjecture is true.
Assuming that $3$, $5$, $17$, $257$, and $65537$ are the only Fermat's primes, there are only eight Carmichael numbers with $L=2^{\alpha}P$, and $P$ is one of $3$, $5$, $7$ or $127$.
The prime factorizations of these numbers are given in \cite{MR2988558}.

In this article, we extend Wright's result to the next simplest case $L=2^{\alpha}P^2$. 
Let $m$ be a Carmichael number with $L=2^{\alpha}P^2$.
Then each prime factor of $m$ is one of the primes of the form $p=2^k+1$, $q=2^lP+1$, or $r=2^sP^2+1$.
(We call these prime factors \textit{Type 1}, \textit{Type 2}, and \textit{Type 3} primes, respectively.)
We assume that $m$ is divisible by at least one of the known Fermat prime numbers.
Under this assumption, we prove the following theorem.
\begin{theorem}\label{thm:main theorem}
	Let $m$ be a Carmichael number with $L=2^{\alpha}P^2$ for some odd prime $P$.
	If $m$ is divisible by one of the known Fermat primes, then $m$ must be one of the following 11 Carmichael numbers. In particular, $P$ is either $5$ or $3$.
	
\begin{align*}
	3\cdot 11 \cdot 17 \cdot 401  \cdot 641 \cdot 1601 && L&=2^2\cdot 5^2\\
	5\cdot 7 \cdot 17 \cdot 19   \cdot 73 && L&=2^4\cdot 3^2\\
	 5\cdot 7 \cdot 19 \cdot 73 \cdot 97 \cdot 257 && L&=2^8\cdot 3^2\\
		5\cdot 13 \cdot 257 \cdot 577 \cdot 1153 && L&=2^8\cdot 3^2\\
		5 \cdot 37 \cdot 73\cdot 257 \cdot 577 \cdot 769 && L&=2^8\cdot 3^2\\
		5 \cdot 37 \cdot 73 \cdot 193 \cdot 257\cdot 1153 && L&=2^8\cdot 3^2\\
	5 \cdot 13 \cdot 257\cdot 577 \cdot 1153 \cdot 18433 && L&=2^{11}\cdot 3^2\\
	5\cdot 257 \cdot 37 \cdot 73 \cdot 577  \cdot 12289 \cdot 18433 && L&=2^{12}\cdot 3^2\\
	5\cdot 7  \cdot 13 \cdot 19 \cdot 37 \cdot 73 \cdot 193 \cdot 257 \cdot 577 \cdot 1153 && L&=2^8\cdot 3^2\\
	5\cdot 7  \cdot 13 \cdot 19\cdot 37 \cdot 73 \cdot 257
 \cdot 577 \cdot 769 \cdot 1153 && L&=2^8 \cdot 3^2\\
 5\cdot 7 \cdot 19 \cdot 73 \cdot 97 \cdot 193 \cdot 577 \cdot 769 \cdot 12289 \cdot 147457 \cdot 65537 \cdot 1179649 \cdot 786433 && L&=2^{18}\cdot 3^2\\
\end{align*}
\end{theorem}

In Section \ref{sec:general result}, we review a theorem (Theorem \ref{thm:power of 2}) that played an important role in \cite{MR2988558} and will be used extensively in this article as well. Suppose that $n=2^{\beta}x+1$ is an integer, where $x$ is odd. Then we call the exponent $\beta$ of $2$ the \textit{$2$-power} of $n$.
Theorem \ref{thm:power of 2} proves that if a Carmichael number is expressed as the product of several odd integers, then there cannot be a unique smallest $2$-power among these integers.

Several lemmas will be given in Section \ref{sec:lemmas} that narrow down the number of possible primes $P$.
In particular, we will see that a Carmichael number under our assumption must have at least two Fermat primes, and $P$ is a divisor of $R-1$, where $R$ is the product of Fermat prime factors of $m$.

Section \ref{subsec:procedure} provides a general procedure to obtain all Carmichael numbers for a given $P$.
Theorem \ref{thm:minimality}, which will be referred to as the minimality argument, implies the existence of some prime factor or a pair of prime factors of a Carmichael number with relatively small $2$-power.
As there are not so many possible such prime numbers or pairs, the procedure terminates and produces all Carmichael number for $P$, or proves that there are no Carmichael numbers for $P$.

By Theorem \ref{thm:power of 2}, at least two of the smallest $2$-powers of Type 1, 2, and 3 primes must be the same. 
Hence there are three cases to consider according to which two $2$-powers are the same.

The rest of the paper will be devoted to a careful scrutiny of Carmichael numbers in these three cases.

\section*{Acknowledgement}
I would like to show my gratitude to Samuel Wagstaff for his suggestion to extend the result of Wright and for his valuable comments.
I would also like to thank Byron Heersink for many helpful conversations.
\section{General results on Carmichael numbers}\label{sec:general result}

Let $m$ be a Carmichael number. Let $L$ be the least common multiple of $p-1$, where $p$ runs over the prime factors of $m$.
Korselt's criterion yields that $L$ divides $m-1$.

The following result is crucial, which is proved in \cite[Theorem 2.1]{MR2988558}. 
For completeness, we give its proof.

\begin{theorem}\label{thm:power of 2}
Let $m$ be a Carmichael number and write 
\[m=\prod_{i=1}^n(2^{\alpha_i}D_i+1),\]
where $D_i$ are odd integer, $n\geq 2$, and $\alpha_1 \leq \alpha_2 \leq \dots \leq \alpha_n$.
Then if $2^{\alpha_1} \mid  L$, then $\alpha_1=\alpha_2$.
\end{theorem}

\begin{proof}
Seeking a contradiction, assume that $\alpha_1 < \alpha_2$.
Then we have
\[m=\prod_{i=1}^n(2^{\alpha_i}D_i+1) \equiv 2^{\alpha_1}D_1+1 \pmod{2^{\alpha_1+1}}.\]

Since $2^{\alpha_1+1}\mid L$ by assumption and $L\mid m-1$ by Korselt's criterion, we obtain
\[m \equiv 1 \pmod{2^{\alpha_1+1}}.\]
It follows that 
\[2^{\alpha_1}D_1+1 \equiv 1 \pmod{2^{\alpha_1+1}}.\]
However, this implies that $D_1$ is even, which contradicts that $D_1$ is odd.
Thus, we have $\alpha_1=\alpha_2$.
\end{proof}
Observe that Theorem \ref{thm:power of 2} does not assume that each factor $2^{\alpha_i}D_i+1$ to be prime.

\section{Lemmas}\label{sec:lemmas}
In this section, we prove several lemmas that reduce the possible factors of a Carmichael numbers.

Let $m$ be a Carmichael number. Let $L$ be the least common multiple of $p-1$, where $p$ runs over the prime factors of $m$.
In this paper, we assume that
\[L=2^{\alpha}P^2,\]
where $\alpha\in \N$ and $P$ is an odd prime number.
This implies that each of the prime factors of $m$ is one of the followings:
\begin{align*}
&p=2^{k}+1 && \text{Type 1 (Fermat's prime)}\\
&q=2^lP+1 && \text{Type 2}\\
&r=2^sP^2+1 && \text{Type 3}.
\end{align*}

We further assume that $m$ has at least one Type 1 prime factor.

From now on, we reserve the letters $p, q, r$ (and the versions with subscripts $p_i, q_i, r_i$) for Type 1, 2, 3 primes, respectively.
Similarly we reserve letters $k, l, s$ for the exponent of $2$ in these prime numbers.

When we deal with numbers of the form $2^{\alpha}D+1$ with odd $D$, then we simply call $\alpha$ \textit{the $2$-power}.
So for example, the $2$-power of $41=2^8\cdot 5+1$ is $8$.

\begin{lemma}\label{lem:mod 3}
Let $P$ be an odd prime number. Suppose that $2^k+1$, $2^lP+1$, and $2^sP^2+1$  are prime numbers. Then:
\begin{enumerate}
\item $k$ is a power of $2$.
\item If $P\equiv 1 \pmod{3}$, then $l$ is even.
\item If $P\equiv 2 \pmod{3}$, then $l$ is odd.
\item If $P\neq 3$, then $s$ is even.
\end{enumerate}
\end{lemma}

\begin{proof}
The first three statements are proved in \cite[Lemma 3.1]{MR2988558}.

Suppose that $P\neq 3$. Then we have $P^2\equiv 1 \pmod{3}$.
Then
\begin{align*}
2^sP^2+1\equiv2^s+1&\pmod{3}\\
\equiv \begin{cases}
2 & \text{if $s$ is even}\\
0 & \text{if $s$ is odd}.
\end{cases}
\end{align*}
\end{proof}

In the next lemma, a Carmichael number $m$ might or might not have Type 1 prime factors.
\begin{lemma}\label{lem:P not congruent to 1 mod 12}
Let $m$ be a Carmichael number with $L=2^{\alpha}P^2$.
Suppose $m\neq p_1q_1r_1$ with $k_1=l_1=s_1$.
If $P \equiv 1 \pmod{3}$, then $P\not \equiv 1 \pmod{4}$.
\end{lemma}

\begin{proof}
	Let
	\[m=\prod_{i \geq 1}v_i,\]
	where $v_i$ is a prime of Type 1, 2, or 3.
	Note that any Carmichael has three or more prime factors.
	Let $\beta_i$ be the $2$-power of $v_i$ and suppose that $0< \beta_1= \beta_2 \leq \beta_3 \leq \cdots$.
	Observe that $\beta_1=\beta_2$ here by Theorem \ref{thm:power of 2}.
	Note that since $P\equiv 1 \pmod{3}$, both $l_i$ and $s_i$ are even by Lemma \ref{lem:mod 3}.
	Hence $k\neq 1$, otherwise $k_1$ is the unique smallest power of $2$ and this contradicts Theorem \ref{thm:power of 2}. So $k_i$ is even, and hence every $\beta_i$ is even.
Let 
\[v_i=2^{\beta_i}P^{\delta_i}+1,\]
where $\delta_i=0, 1, 2$ depending on the type of $v_i$.
Then we have
\begin{align*}
v_1v_2&=(2^{\beta_1}P^{\delta_1}+1)(2^{\beta_1}P^{\delta_2}+1)\\
&=2^{2\beta_1}P^{\delta_1+\delta_2}+2^{\beta_1}P^{\delta_1}+2^{\beta_1}P^{\delta_2}+1\\
&=2^{\beta_1}(2^{\beta_1}P^{\delta_1+\delta_2}+P^{\delta_1}+P^{\delta_2})+1.
\end{align*}

Now seeking a contradiction, assume that $P\equiv 1 \pmod{4}$.
Then we have
\[2^{\beta_1}P^{\delta_1+\delta_2}+P^{\delta_1}+P^{\delta_2}\equiv 2 \pmod{4}\]
since $\beta_1$ is even.
It follows that the $2$-power of $v_1v_2$ is $\beta_1+1$.

Then consider the expression
\[m=(v_1v_2)\prod_{i \geq 3}v_i.\]
If $\beta_1 < \beta_3$, then $\beta_3 \geq \beta_1+2$ as each $\beta_i$ is even.
Thus $2^{\beta_1+2}|L$ and $2^{\beta_1+1}$ is the unique smallest power of $2$ in the above expression of $m$.
This contradicts Theorem \ref{thm:power of 2}.
Hence $\beta_3=\beta_1$. Then we have $v_1v_2 v_3=p_1q_1r_1$ with $k_1=l_1=s_1$, and by assumption that $m\neq p_1q_1r_1$, there exists $v_4$ with $\beta_4 \geq \beta_1+2$.
Then in the expression
\[m=v_3(v_1v_2)\prod_{i\geq 3}v_i,\]
$v_3$ has the unique smallest power $\beta_1$ of $2$, and $2^{\beta_1+1}\mid 2^{\beta_4}\mid L$.
This contradicts Theorem \ref{thm:power of 2}.
Hence we conclude that $P\not \equiv 1 \pmod{4}$.

\end{proof}

\subsection{Fermat primes and the prime $P$}

\begin{lemma}\label{lem:P divides product of Type 1}Let $m$ be a Carmichael number with $L=2^{\alpha}P^2$.
	Let $p_1, \dots, p_n$ be the Type 1 prime divisors of $m$. Then
	\[p_1\cdots p_n \equiv 1 \pmod{P}.\]
\end{lemma}
\begin{proof}
	By Korselt's criterion, we have $P\mid L\mid m-1$. It follows that
	\[1\equiv m \equiv p_1\cdots p_n \pmod{P}\]
	since Type 2 and Type 3 primes are congruent to $1$ modulo $P$.
\end{proof}

\begin{corollary}\label{cor:two Fermat's prime}
	Let $m$ be a Carmichael number with $L=2^{\alpha}P^2$.
	Assume that $m$ has a Type 1 (Fermat) prime factor. Then $m$ has at least two distinct Fermat prime factors.
\end{corollary}
\begin{proof}
	If there is a unique Type 1 divisor $p_1=2^{k_1}+1$ of $m$, then Lemma \ref{lem:P divides product of Type 1} yields that 
	\[p_1=2^{k_1}+1\equiv 1 \pmod{P}.\]
	Thus $P\mid 2^{k_1}$ and this is impossible because $P$ is an odd prime.
	
\end{proof}

For the rest of the paper, we assume that $m$ is a Carmichael number with $L=2^{\alpha}P^2$ and $m$ has a Type 1 factor. Also we restrict Type 1 factors to be known Fermat's primes: $p_i=3, 5, 17, 257, 65537$.

By Lemma \ref{lem:P divides product of Type 1} and Corollary \ref{cor:two Fermat's prime}, the prime $P$ appears in the prime factorization of $R-1$, where $R$ is the product of two or more known Fermat's primes.

Table \ref{table:products of fermat primes} can be found in \cite[Table 1]{MR2988558}, though we rearranged the table by $k_1$.

\begin{table}[ht]
\caption{Products of Fermat primes}
\centering
\begin{tabular}{c c c c}
\hline\hline
Combination of Primes $(R)$ & Factorization of $R-1$ & $k_1$  \\ [0.5ex] 
\hline
$3\ast 5$ & $2\ast7$ & $1$ &\\
$3\ast 17$ & $2\ast 5^2$ & $1$ &\\
$3\ast 257$ & $2\ast 5 \ast 7 \ast 11$ & $1$ &\\
$3\ast 65537$ & $2\ast 5 \ast 19661$ & $1$ &\\
$3\ast 5 \ast 17$ & $2\ast 127$ & $1$ &\\
$3\ast 5 \ast 257$ & $2\ast 41 \ast 47 $ & $1$ &\\
$3\ast 17 \ast 257$ & $2\ast 6553$ & $1$ &\\
$3\ast 5 \ast 65537$ & $2\ast 491527$ & $1$ &\\
$3\ast 17 \ast 65537$ & $2\ast 127 \ast 13159$ & $1$ &\\
$3\ast 257 \ast 65537$ & $2\ast 25264513$ & $1$ &\\
$3\ast 5 \ast 17 \ast 257$ & $2\ast 7 \ast 31 \ast 151$ & $1$ &\\
$3\ast 5 \ast 257 \ast 65537$ & $2\ast 7 \ast 18046081$ & $1$ &\\
$3\ast 5 \ast 17 \ast 65537$ & $2\ast 8355967$ & $1$ &\\
$3 \ast 17 \ast 257 \ast 65537$ & $2\ast 19 \ast 22605091$ & $1$ &\\
$3 \ast 5 \ast 17 \ast 257 \ast 65537$ & $2\ast 2147483647$ & $1$ &\\
\hline
$5\ast 17$ & $2^2 \ast 3 \ast 7$ & $2$ &\\
$5\ast 257$ & $2^2 \ast 3 \ast 107$ & $2$ &\\
$5\ast 65537$ & $2^2 \ast 3 \ast 7 \ast 47 \ast 83$ & $2$ &\\
$5\ast 17 \ast 257$ & $2^2 \ast 43 \ast 127$ & $2$ &\\
$5\ast 17 \ast 65537$ & $2^2 \ast 131 \ast 10631$ & $2$ &\\
$5\ast 257 \ast 65537$ & $2^2 \ast 467 \ast 45083$ & $2$ &\\
$5 \ast 17 \ast 257 \ast 65537$ & $2^2 \ast 3 \ast 7 \ast 11 \ast 31 \ast 151 \ast 331$ & $2$ &\\
\hline
$17\ast 257$ & $2^4 \ast 3^3 \ast 7 \ast 13$ & $4$ &\\
$17\ast 65537$ & $2^4 \ast 3^3 \ast 2579$ & $4$ &\\
$17 \ast 257 \ast 65537$ & $2^4 \ast 29 \ast 43 \ast 113 \ast 127$ & $4$ &\\
\hline
$257 \ast 65537$ & $2^8 \ast 3 \ast 7 \ast 13 \ast 241$ & $8$ &\\[1ex]
\hline
\end{tabular}
\label{table:products of fermat primes}
\end{table}

\subsection{Cases}
Let
\[m=\left(\prod_{i=1}^{n_1}p_i\right)\left(\prod_{i=1}^{n_2}q_i\right)\left(\prod_{i=1}^{n_3}r_i\right)=\left(\prod_{i=1}^{n_1} 2^{k_i}+1 \right)\left(\prod_{i=1}^{n_2} 2^{l_i}P+1\right)\left(\prod_{i=1}^{n_3} 2^{s_i}P^2+1\right)\]
with $k_1<k_2<\dots < k_{n_1}$, $l_1< l_2 <\dots < l_{n_2}$, and $s_1<s_2< \dots<s_{n_3}$
be a Carmichael number with $L=2^{\alpha}P^2$.
We are assuming $n_1 \geq 1$.
By Theorem \ref{thm:power of 2}, there cannot be a unique smallest $2$-powers.
Thus, there are three cases to consider.
\begin{enumerate}
\item[\textbf{Case A}] $k_1=l_1 \leq s_1$.
\item[\textbf{Case B}] $k_1=s_1< l_1$.
\item[\textbf{Case C}] $l_1=s_1 <k_1$.
\end{enumerate}

\section{Procedure}\label{subsec:procedure}
We explain the procedure to find all Carmichael numbers for a given $P$.
In the sequel, we use the letter $x$ to denote an odd number but its actual value could be different in each occurrence. 
For example, we write
\[5\cdot 13=(2x+1)(2^2x+1)=2^6x+1.\]
In this case, the actual values are $x=2, 3, 1$ in this order.

Let us fix $P$. Then each of Case A, B, C, we start with two prime numbers with minimal $2$-power together with Type 1 primes that give $P$.

\begin{enumerate}
\item[\textbf{Step 1}]
We multiply some or all of these known factors and write it as $(2^{a_1}x+1)\cdots (2^{a_n}x+1)$, with $a_1< a_2 \leq a_3 \leq \cdots \leq a_n$. It is possible that we have only one term.
Then a Carmichael number is of the form
\[m=(2^{a_1}x+1)\cdots (2^{a_n}x+1)\prod_{i \geq 1} v_i,\]
where $v_i$ are primes of Type 1, 2, 3, or the product could be empty.
If the product is empty, then $m=(2^{a_1}x+1)\cdots (2^{a_n}x+1)$ is a Carmichael number, otherwise there is no Carmichael number for this $P$.

\item[\textbf{Step 2}](The minimality argument)
If the $2$-powers of every $v_i$ is greater than $a_1$, then $a_1$ is the unique smallest $2$-power. 
This is prohibited by Theorem \ref{thm:power of 2}.
Thus, there must be $v_i$ of the $2$-power less than or equal to $a_1$.
Since there cannot be a unique smallest $2$-power, we have two cases to consider.
The first case is that there is a pair, say, $(v_1, v_2)$ and the $2$-power of $v_1, v_2$ are the same, say $a$, and $a < a_1$. (We call such a pair \textit{$a$-pair}.)
The second case is that the $2$-power of $v_1$ is $a_1$, we call such a prime number $\textit{$a_1$-prime}$.

In either case, there are finitely many possible cases.

\item[\textbf{Step 3}] We multiply some or all of the factors we obtained in Step 1.
Then we have
\[m=(2^{b_1}x+1)\cdots (2^{b_f}x+1)\prod v_i,\]
where $v_i$ can be a prime of Type 1, 2, 3 that did not appear in Step 1 (recall that every Carmichael number is squarefree), or the product $\prod v_i$ could be empty.

We repeat Step 2 and 3 until there is no possible pair $(v_i, v_{i+1})$ or $v_j$ in Step 2.

\end{enumerate}

To narrow down the number of possible prime factors of a Carmichael number in the procedure, the following theorem will be useful. The theorem is true for any Carmichael numbers without any restrictions.

\begin{theorem}[The minimality argument]\label{thm:minimality}
	
	Let $m$ be a Carmichael number. Suppose that $2^a\mid L$ for some $a\in \N$.
	Write
	\[m=(2^bx+1)(2^{b_1}x_1+1)\cdots (2^{b_f}x_f+1)\prod_{i=1}^{g}v_i,\]
	where $x, x_1, \dots, x_f$ are odd integers and $v_i$ is a prime factor of $m$.
	Let $\beta_i$ be the $2$-power of $v_i$. That is, $2^{\beta_i}D+1=v_i$ for some odd $D$.
	Assume that
	\[b < b_1\leq b_2 \leq \cdots \leq b_f \text{ and } \beta_1\leq \beta_2 \leq \cdots \leq \beta_g.\]
	\begin{enumerate}
		\item If $b < a$, then $v_1$ is a $b$-prime or $(v_1, v_2)$ is $\beta_1$-pair with $\beta_1 < b$.
		\item If $b \geq a$, then either $v_1$ is a $\beta_1$-prime with $a\leq \beta_1 \leq b$ or we have a $\beta_1$-pair $(v_1, v_2)$ with $\beta_1 <a$.		
	\end{enumerate}
	\end{theorem}
	
	\begin{proof}
		\begin{enumerate}
			\item Since $b < a$, we have $2^{b+1}\mid 2^a\mid L$.
			It follows from Theorem \ref{thm:power of 2} that $b$ cannot be the unique smallest $2$-power. 
			Thus, we have the following possibilities: $b=\beta_1$, $\beta_1=\beta_2=b$, or $\beta_1=\beta_2 < b$ by Theorem \ref{thm:power of 2}.
			The first two cases yield that $v_1$ is a $b$-prime.			The third case implies that $(v_1, v_2)$ is a $\beta_1$-pair and $\beta_1 < b$.
			
			\item Assume that $\beta_1 > b$. Then we have $2^{b+1}\mid 2^{\beta_1}\mid L$.
			This implies that $b$ is the unique smallest $2$-power, which contradicts Theorem \ref{thm:power of 2}. Thus we have $\beta_1\leq b$.

			If $\beta_1 <a$, then $\beta_1<a\leq b <b_1$. 
			Write $m$ as
			\[m=(2^{\beta_1}+1)(2^bx+1)(2^{b_1}x_1+1)\cdots (2^{b_f}x_f+1)\prod_{i=2}^{g}v_i.\]
			Note that the product $\prod_{i=2}^{g}v_i$ cannot be empty, otherwise $\beta_1$ is the unique smallest $2$-power and this contradicts Theorem \ref{thm:power of 2}.
			
			Since the product is nonempty, we can apply part (1) with $\beta_1$ instead of $b$.
			Thus either $v_2$ is a $\beta_1$-prime or $(v_2, v_3)$ is a $\beta_2$-pair with $\beta_2 < \beta_1$. But the latter case never happen as $\beta_1\leq \beta_2$.
			Hence $(v_1, v_2)$ is a $\beta_1$-pair, with $\beta_1 <a$.								\end{enumerate}
	\end{proof}


\section{Case A: $k_1=l_1 \leq s_1$.}
We classify Carmichael numbers with $k_1=l_1 \leq s_1$ in this section.
We reduce the number of the possible $P$ from Table \ref{table:products of fermat primes} by removing those $(P, k_1)$ such that $2^{k_1}P+1$ are not prime.
This procedure is done in \cite[Table 2]{MR2988558} except that we do not have exact counterparts of \cite[Theorem 3.2 and Theorem 3.3]{MR2988558}.
The lack of \cite[Theorem 3.3]{MR2988558} leads two additional pairs $(P, k_1)=(11, 1), (41, 1)$ to consider.
Also we need to consider pairs $(13, 4)$, $(13, 8)$ because of the lack of \cite[Theorem 3.2]{MR2988558}.
However, the latter two pairs can be eliminated as follows.
Since $13\equiv 1 \pmod{12}$, by Lemma \ref{lem:P not congruent to 1 mod 12}, we must have $m=p_1q_1r_1$ with $k_1=l_1=s_1$.
For the pair $(13, 4)$, the number $2^4\cdot 13+1=11\cdot 19$ is not prime, and for the pair $(13, 8)$ the number $2^8\cdot 13^2+1=5\cdot 17 \cdot 509$ is not prime. Hence these pairs do not produce a Carmichael number.

Table \ref{table:Candidates for P} lists possible candidates for $P$.
Note that $11$ is removed from $5 \ast 17 \ast 257 \ast 65537$ since $k_1=l_1$ is odd as $11\equiv 2 \pmod{3}$.
\begin{table}[ht]
\caption{After removing composite $2^{k_1}P+1$}
\centering
\begin{tabular}{c c c c}
\hline\hline
Combination of Primes $(R)$ & Possible factors of $R-1$ & $k_1$  \\ [0.5ex] 
\hline
$3\ast 17$ & $5$ & $1$ &\\
$3\ast 257$ & $ 5, 11$ & $1$ &\\
$3\ast 65537$ & $5, 19661$ & $1$ &\\
$3\ast 5 \ast 257$ & $41$ & $1$ &\\
\hline
$5\ast 17$ & $3, 7$ & $2$ &\\
$5\ast 257$ & $3$ & $2$ &\\
$5\ast 65537$ & $3, 7$ & $2$ &\\
$5\ast 17 \ast 257$ & $43, 127$ & $2$ &\\
$5 \ast 17 \ast 257 \ast 65537$ & $3, 7$ & $2$ &\\[1ex]
\hline
\end{tabular}
\label{table:Candidates for P}
\end{table}

Thus the possible primes are
\[P=3, 5, 7, 11, 41, 43, 127, 19661.\]

In the sequel, we actually prove that none of these produces a Carmichael number except for $P=3, 5$.

\subsection{The impossible case: $P=43$}
As in \cite{MR2988558}, let us start with $43$.
We prove that there is no Carmichael number with $P=43$.
This case illustrates the procedure explained in Section \ref{subsec:procedure}.

Let $m$ be a Carmichael number with $L=2^{\alpha}\cdot 43^2$.
Since $P=43$, the product of Type 1 primes must be $5\cdot 17 \cdot 257$ and from Table \ref{table:Candidates for P}.

Since $k_1=l_1=2$ (recall we are dealing with Case A), we have
\[q_1=2^2\cdot 43+1=173.\]
Let
\[m=5\cdot 17 \cdot 257 \cdot 173\prod_{i=1}u_i,\]
where $u_i$ is either Type 2 or Type 3 prime (and there must be at least one Type 3 prime.)
Then we have
\begin{align*}
m=(2^4+1)(2^5\cdot 27+1)(2^8+1)\prod_{i=1}u_i.
\end{align*}
Since $2^8\mid L$, the product cannot be empty, otherwise $4$ is the unique smallest $2$-power.

Thus, we have two cases: There exists a pair $(u_1, u_2)$ of Type 2 and Type 3 primes with $2$-power less than $4$ (we call such a pair \textit{$n$-pair} with $n<4$), or there exists $u_1$ of Type 2 or Type 3 with $2$-power $4$ (we call such a prime \textit{4-prime}).
Since $k_1=l_1=2$ is used already, the $2$-power of a pair must be greater than $2$ and less than $4$.
Thus the $2$-power of a pair must be $3$. 
However, since $43\equiv 1 \pmod{3}$, we know $l_i$ is even by Lemma \ref{lem:mod 3}.
So there is no such pair.

Note that the numbers
\[2^4\cdot 43+1=13\cdot 53 \text{ and } 2^4 \cdot 43^2+1=5\cdot 61 \cdot 97\]
are composite.
So there is no $4$-prime of Type 2 or 3.

We proved:
\begin{theorem}
	If $P=43$, then there is no Carmichael number in Case A.
\end{theorem}

\subsection{The impossible case: $P=19661$}
Let us next consider the case $P=19661$.
From Table \ref{table:Candidates for P}, the product of Type 1 primes is $3\cdot 65537$ and $k_1=1$.
Since $k_1=l_1=1$, we have $q_1=2\cdot 19661+1=39323$.
Let
\[m=3\cdot 65537 \cdot 39323 \prod_{i\geq 1}u_i,\]
where $u_i$ is a prime of Type 2 or Type 3.
Note that $2^{16}\mid L$, and $P\equiv 2 \pmod{3}$, hence $l_i$ is odd by Lemma \ref{lem:mod 3}.
We have
\[m=(2^4 \cdot 7373+1)(2^{16}+1)\prod_{i \geq 1}u_i.\]
This implies that there exists a pair $(u_1, u_2)$ of Type 2 and 3 with $2$-power between $k_1=l_1<2$ and $3$, or there exists $4$-prime of Type 2 or 3.
Since $l_i$ is odd, the $2$-power of such a pair must be $3$, and $4$-prime must be Type 3.
Note that the numbers
\begin{align*}
2^3 \cdot 19661+1&=11\cdot 79 \cdot 181\\
2^4\cdot 19661^2+1&=3217\cdot 1922561
\end{align*}
are composite, hence there is no $3$-pair or $4$-prime.

This proves:
\begin{theorem}
	If $P=19661$, then there is no Carmichael number for Case A.
\end{theorem}

\subsection{The impossible case: $P=41$}
If $P=41$, then the product of Type 1 primes is $3\cdot 5 \cdot 257$ from Table \ref{table:Candidates for P} and $k_1=l_1=1$.
So $q_1=2\cdot 41+1=83$. Note $2^8\mid L$ and $P\equiv 2 \pmod{3}$, hence $l_i$ is odd.

Let
\[m=3\cdot 5 \cdot 257\cdot 83 \prod_{i\geq 1}u_i,\]
where $u_i$ is a prime of Type 2 or Type 3.
Then we have
\[m=(2^2+1)(2^3\cdot 31+1)(2^8+1)\prod_{i \geq 1}u_i.\]
Since $s_i$ is even (Lemma \ref{lem:mod 3}), no smaller $2$-power pair $(u_1, u_2)$.
Since $l_i$ is odd, the only candidate for $2$-prime is the Type 3 $2$-prime.
However, the number
\[2^2\cdot 41^2+1=5^2\cdot 269,\]
 is not prime.
Hence we have:
\begin{theorem}
	If $P=41$, then there is no Carmichael number for Case A.
\end{theorem}

\subsection{The impossible case: $P=11$}
When $P=11$, the product of Type 1 primes is $3\cdot 257$ from Table \ref{table:Candidates for P}. Then we have $k_1=l_1=1$, and thus $q_1=2\cdot 11+1=23$.
Using $\prod_{i \geq }{u_i}$ as in the previous cases, we have
\begin{align*}
m&=3\cdot 257 \cdot 23 \prod_{i \geq 1} u_i\\
&=(2\cdot 17+1)(2^8+1)\prod_{\i \geq 1} u_i.
\end{align*}
Since $s_1$ is even by Lemma \ref{lem:mod 3}, there is no $1$-prime among $u_i$.
Thus, $2^1$ is the unique smallest power and $2^8\mid L$, we conclude that $m$ cannot be a Carmichael number in this case. 
So we have:
\begin{theorem}
	If $P=11$, then there is no Carmichael number for Case A.
\end{theorem}


\section{The impossible case: $P=7$}
We still assume $k_1=l_1$ (Case A) and prove that there is no Carmichael number with $P=7$.
From Table \ref{table:Candidates for P}, there are three possibilities for the Type 1 primes: 
\[5\cdot 17, \qquad 5\cdot 65537, \qquad 5 \cdot 17 \cdot 257 \cdot 65537.\]
In any case, $k_1=l_1=2$, and thus $q_1=2^2\cdot 7+1=29$.

Let us first deal with the case when $17$ divides $m$.
\subsection{Case 1: $17$ divides $m$.}
Suppose $17$ divides $m$.
Let
\[m=5\cdot 17 \cdot 29 \prod_{i\geq 1}v_i,\]
where $v_i$ is a prime of Type 1, 2, or 3.
The only possible Type 1 primes for $v_i$ are $257$ and $65537$.
Then we have
\begin{align*}
m=(2^5\cdot 77+1)\prod_{i \geq 1}v_i.
\end{align*}
Note that $2^4\mid L$. By the minimality argument (Theorem \ref{thm:minimality}), either we must have an $n$-pair $(v_1, v_2)$ with $n< 4$, or an $n$-prime with $n=4, 5$.
Table \ref{table:P=7} lists data of the numbers of the form $2^l\cdot 7+1$, $2^s\cdot 7^2+1$, and $2^k+1$.
Note that since $7\equiv 1 \pmod{3}$, both $l_i, s_i$ are even by Lemma \ref{lem:mod 3}.
From the table, we see that there is no such a pair.
As $17$ is already used, we have a unique $4$-prime: $113$.

Hence we have
\begin{align*}
m=5 \cdot 17 \cdot 29 \cdot 113 \prod_{i \geq 2}v_i=(2^4\cdot 3\cdot 7 \cdot 829+1)\prod_{i \geq 2}v_i.
\end{align*}

Since the product still must contain a Type 3 prime, it is not empty.
By the minimality argument, there must be an $n$-pair with $n<4$ or $4$-prime, but this is impossible from Table \ref{table:P=7} as there is no more $4$-primes.

\begin{table}[ht]
\caption{Primality for $2^l\cdot 7+1$, $2^s\cdot 7^2+1$, $2^k+1$}
\centering
\begin{tabular}{|c| c| c |c |c}
\hline\hline
$n$ & $2^n\cdot 7+1$ & $2^n\cdot 7^2+1$ & $2^n+1$  \\ [0.5ex] 
\hline
$2$ & $29$ prime & $197$ prime & $5$ prime  \\
$4$ & $113$ prime & $\times$ & $17$ prime \\
$6$ & $449$ prime & $3137$ prime& $\times$ \\
$8$ &  $\times$ & $\times$ & $257$ prime \\
$10$ &  $\times$ & $50177$ prime & $\times$ \\
$12$ &  $\times$ & $\times$  & $\times$ \\
$14$ &  $114689$ prime & $\times$  & $\times$ \\
$16$ &  $\times$ & $\times$  & 65537\\[1ex]
\hline
\end{tabular}

$\times$=composite, $7\equiv 1\pmod{3}$
\label{table:P=7}
\end{table}

\subsection{Case 2: $17$ does not divide $m$.}
Suppose that $m$ is not divisible by $17$. 
Namely, the case when the type 1 combination is $5\cdot 65537$.
Then we have $k_1=l_1=2$, and $q_1=2^2\cdot 7+1=29$.
Let
\[m=5\cdot 65537 \cdot 29 \prod_{i \geq 1}u_i,\]
where $u_i$ is a prime of Type 2 or 3.
We have $2^{16}\mid L$.
Then 
\[m=(2^4\cdot 9+1)(2^{16}+1)\prod_{i \geq 1}u_i.\]
From Table \ref{table:P=7}, we see that there is no pair of $u_i$ of $2$-power smaller than $4$.
The only $4$-prime $u_1$ is $q_2=113$ with $l_2=4$ (as we assume $17\nmid m$).
Since $5\cdot 29\cdot 113=2^{14}+1$, we have
\begin{align*}
m=(2^{14}+1)(2^{16}+1)\prod_{i \geq 2}u_i
\end{align*}
By the minimality argument, we must have either a pair $(u_2, u_3)$ of Type 2 and 3 with $2$-power between $5$ and $13$, or a $14$-prime.
From Table \ref{table:P=7}, we see that $(449, 3137)$ is $6$-pair and $114689$ is a $14$-prime.

\subsubsection{Subcase 1: $(u_2, u_3)=(449, 3137)$}
If $(u_2, u_3)=(449, 3137)$, then we have
\begin{align*}
m&=5\cdot 65537 \cdot 29 \cdot 113 \cdot 449 \cdot 3137\prod_{i \geq 4}u_i\\
&=(2^9\cdot 45075167+1)(2^{16}+1)\prod_{i \geq 4}u_i.
\end{align*}
If the product is empty, then $2^9$ is the unique smallest power of $2$, and this contradicts Theorem \ref{thm:power of 2}.
Thus the product is nonempty.
Since there is no $9$-prime, and there is no $n$-pair for $7\leq n \leq 8$ from Table 
\ref{table:P=7}, this case does not happen.

\subsubsection{Subcase 2: $u_2=q_3=114689$.}
If $u_2=q_3=114689$, then we have
\begin{align*}
m&=5\cdot 65537 \cdot 29 \cdot 113 \cdot 114689 \prod_{i \geq 3}u_i\\
&=(2^{16}+1)(2^{17}\cdot3^5\cdot 59+1)\prod_{i \geq 3}u_i.
\end{align*}
Note that the product is not empty since it contains a Type 3 prime.
There is no $16$-prime from Table \ref{table:P=7}.
The only $n$-pair is $(449, 3137)$ with $n<16$ but we dealt with this case in Subcase 1.

Thus, there is no Carmichael numbers for Case 2 as well.
This completes the proof of:
\begin{theorem}
	If $P=7$, then there is no Carmichael number for Case A.
\end{theorem}

\section{The impossible case: $P=127$}
Next, we prove that there is no Carmichael number when $P=127$ for $k_1=l_1 \leq s_1$ (Case A).

From Table \ref{table:Candidates for P}, there is only one Type 1 combination: $5\cdot 17 \cdot 257$ with $k_1=l_1=2$.
So we have $q_1=2^2\cdot 127+1=509$.
Let
\[m=5\cdot 17 \cdot 257 \cdot 509 \prod_{i \geq 1} u_i,\]
where $u_i$ is a prime of Type 2 or Type 3.
We have $2^8\mid L$.
Then we have
\[m=(2^9\cdot 3^2\cdot 19 \cdot 127+1)\prod_{i \geq 1} u_i.\]

Since $P=127\equiv 1 \pmod{3}$, we know that $l_i, s_i$ are even.
Since the product contains a Type 3 prime, it is not empty.
By the minimality argument (Theorem \ref{thm:minimality}), there is an $n$-pair with $n<8$ or an $n$-prime with $n=8, 9$.
From Table \ref{table:P=127}, we see that there is no such an $n$-pair or an $n$-prime.
Thus, $m$ cannot be a Carmichael number.

\begin{table}[ht]
\caption{Primality for $2^l\cdot 127+1$, $2^s\cdot 127^2+1$}
\centering
\begin{tabular}{|c| c| c |c }
\hline\hline
$n$ & $2^n\cdot 127+1$ & $2^n\cdot 127^2+1$  \\ [0.5ex] 
\hline
$2$ & $509$ prime & $\times$    \\
$4$ & $\times$ & $\times$  \\
$6$ & $\times$ & $\times$ \\
$8$ &  $\times$ & $\times$  \\
$10$ &  $\times$ & $\times$  \\
$12$ &  $520193$ prime & $\times$   \\
$14$ &  $\times$  & $\times$   \\
$16$ &  $\times$ & $\times$  \\[1ex]
\hline
\end{tabular}

$\times$=composite, $127\equiv 1\pmod{3}$
\label{table:P=127}
\end{table}
(We remark that the first two values of $n$ for which $2^n\cdot 127^2+1$ is prime are $n=42, 98$.)

Thus we obtain:
\begin{theorem}
	If $P=127$, then there is no Carmichael number for Case A.
\end{theorem}

\section{Case: $P=5$}
The combinations of Type 1 primes for $P=5$ are the following three from Table \ref{table:Candidates for P}.
\[3 \cdot 17, \qquad 3\cdot 257, \qquad 3 \cdot 65537.\]
So in any case, we have $k_1=l_1=1$, and hence $q_1=2\cdot 5+1=11$.
Since $P=5\equiv 2 \pmod{3}$, we know that $l_i$ is odd and $s_i$ is even.
In all cases, we have $2^4 \mid L$.

Let
\[m=3\cdot 11 \prod_{i \geq 1}v_i=(2^5+1)\prod_{i \geq 1}v_i,\]
where each $v_i$ is a prime of Type 1, 2, or 3.
If $v_i$ is a Type 1 prime, then it must be one of $17, 257, 65537$, and the other $v_j$ is not Type 1.

\begin{table}[ht]
\caption{Primality for $2^l\cdot 5+1$, $2^s\cdot 5^2+1$}
\centering
\begin{tabular}{|c| c| c |c |}
\hline\hline
$n$ & $2^n\cdot 5+1$ & $2^n\cdot 5^2+1$ & $2^n+1$ \\ [0.5ex] 
\hline
$1$ & $11$ prime & $\times$ & $3$ prime    \\
$2$ & $\times$ & $101$ prime  & $5$ prime \\
$3$ & $41$ prime & $\times$  & $\times$\\
$4$ &  $\times$ & $401$ prime & $17$ prime  \\
$5$ &  $\times$ & $\times$  & $\times$\\
$6$ &  $\times$ & $1601$ prime  & $\times$ \\
$7$ &  $641$ prime  & $\times$ & $\times$  \\
$8$ &  $\times$ & $\times$ & $257$ prime \\[1ex]
\hline
\end{tabular}

$\times$=composite, $5\equiv 2\pmod{3}$
\label{table:P=5}
\end{table}

As $2^4 \mid L$, by the minimality argument (Theorem \ref{thm:minimality}), there must be an $n$-pair $n<4$ or an $n$-prime for $n=4, 5$.
From Table \ref{table:P=5}, there is no such an $n$-pair or $5$-prime.
There are two $4$-primes: $401$ and $17$.

\subsection{Case 1: $r_1=401$.}
Let $r_1=401$. Then we have
\[m=3\cdot 11 \cdot 401 \prod_{i \geq 2}v_i=(2^4\cdot 827+1)\prod_{i \geq 2}v_i.\]
If the product is empty, then $m$ is not a Carmichael number since $5\nmid m-1$.
By the minimality argument, there must be the other $4$-prime $17$.
So we have
\begin{align*}
m&=3\cdot 11 \cdot 17 \cdot 401 \prod_{i \geq 3}v_i=(2^6\cdot 5\cdot 19 \cdot 37+1)\prod_{i \geq 3}v_i.
\end{align*}

If the product is empty, then $m$ is not a Carmichael since $5^2 \nmid m-1$.
There is no $n$-pairs for $n <4$ and there is no $n$-prime with $n=4, 5$.
We have a unique $6$-prime: $1601$.
So we have
\begin{align*}
m&=3\cdot 11 \cdot 17 \cdot 401 \cdot 1601 \prod_{i \geq 4}v_i=(2^8\cdot 5\cdot 53 \cdot 5309+1)\prod_{i \geq 4}v_i.
\end{align*}

If the product is empty, then $m$ is not a Carmichael as $5^2 \nmid m-1$.
By the minimality argument, we must have $n$-prime with $n=5, 6, 7, 8$. The only possible case is a $7$-prime: $641$.
We have
\begin{align*}
m&=3\cdot 11 \cdot 17 \cdot 401 \cdot 1601 \cdot 641 \prod_{i \geq 5}v_i\\
&=(2^7\cdot 5^2\cdot 53 \cdot72145063 +1)\prod_{i \geq 5}v_i.
\end{align*}
If the product is empty, then
\[\boxed{3\cdot 11 \cdot 17 \cdot 401  \cdot 641 \cdot 1601}\]
is a Carmichael number since $L=2^7\cdot 5^2$ divides $m-1=2^7\cdot 5^2\cdot 53 \cdot72145063$.
In fact, the product must be empty since we used all $n$-pairs with $n <4$ and $n$-primes for $4\leq n \leq 7$.

\subsection{Case 2: $p_2=17$.}
Let us now consider the case $p_2=17$.
 Then we have
\[m=3\cdot 11 \cdot 17 \prod_{i \geq 2}v_i=(2^4\cdot 5 \cdot+1)\prod_{i \geq 2}v_i.\]
Since there is no possible pair, we must have the other $4$-prime: $401$.
Hence this case reduces to the previous case.

This proves:
\begin{theorem}
	When $P=5$, there is exactly one Carmichael number $m$ with $L=2^{\alpha}\cdot 5^2$ in Case A.
	The factorization of $m$ is given by
	\[m=3\cdot 11 \cdot 17 \cdot 401  \cdot 641 \cdot 1601.\]
\end{theorem}

\section{Case: $P=3$}
In this section, we consider the case A ($k_1=l_1\leq s_1)$ with $P=3$.
From Table \ref{table:Candidates for P}, there are four combinations of Type 1 primes for $P=3$:
\[5\cdot 17, \qquad 5\cdot 257, \qquad 5\cdot 65537, \qquad 5 \cdot 17  \cdot 257 \cdot 65537.\]
In any case, we have $k_1=l_1=2$, and thus $q_1=2^2\cdot 3+1=13$.
We also have $2^{4} \mid 2^{k_2} \mid L$. 
Recall that since $P=3$, the exponents $l_i, s_i$ cannot be restricted by their parities. 

We first deals with $5 \cdot 17$ and $5\cdot 17 \cdot 257 \cdot 65537$ together, and we consider the rest two cases individually.

%

\subsection{Case 1: $17 \mid m$}
Suppose that $17 \mid m$.
Let
\[m=5\cdot 17 \cdot 13 \prod_{i \geq 1} v_i=(2^4\cdot 3\cdot 23+1)\prod_{i \geq 1} v_i,\]
where each $v_i$ is a prime of Type 1, 2, or 3.

\begin{table}[ht]
\caption{Primality for $2^l\cdot 3+1$, $2^s\cdot 3^2+1$}
\centering
\begin{tabular}{|c| c| c |c |}
\hline\hline
$n$ & $2^n\cdot 3+1$ & $2^n\cdot 3^2+1$ & $2^n+1$ \\ [0.5ex] 
\hline
$1$ & $7$ prime & $19$ prime & $3$ prime    \\
$2$ & $13$ prime & $37$ prime  & $5$ prime \\
$3$ & $\times$ & $73$ prime  & $\times$\\
$4$ &  $\times$ & $\times$ & $17$ prime  \\
$5$ &  $97$ prime & $\times$  & $\times$\\
$6$ &  $193$ prime & $577$ prime  & $\times$ \\
$7$ &  $\times$  & $1153$ prime & $\times$  \\
$8$ &  $769$ prime & $\times$ & $257$ prime \\
$9$ &  $\times$ & $\times$ & $\times$  \\
$10$ &  $\times$ & $\times$ & $\times$  \\
$11$ &  $\times$ & $18433$ prime & $\times$  \\
$12$ &  $12289$ prime & $\times$ & $\times$  \\
$13$ &  $\times$ & $\times$ & $\times$  \\
$14$ &  $\times$ & $147457$ prime & $\times$  \\
$15$ &  $\times$ & $\times$ & $\times$  \\
$16$ &  $\times$ & $\times$ & $65537$  prime\\
$17$ &  $\times$ & $1179649$ prime & $\times$  \\
$18$ &  $786433$ prime & $\times$ & $\times$  \\
$|$ &  $\times$  & $\times$ & $\times$  \\
$30$ &  $3221225473$ prime & $\times$ & $\times$  \\[1ex]
\hline
\end{tabular}

$\times$=composite
\label{table:P=3}
\end{table}

After removing the used primes from Table \ref{table:P=3}, we see that there is no $n$-pair with $n<4$.
Also, since $17$ is already used, there is no $4$-prime.
Thus, in this case there is no Carmichael number.

\subsection{Case 2: Type 1 combination is $5\cdot 257$}
We next consider the case when the Type 1 combination is $5\cdot 257$.
Let
\[m=5\cdot 257 \cdot 13 \prod_{i\geq 1} u_i=(2^6+1)(2^8+1) \prod_{i\geq 1} u_i,\]
where each $u_i$ is a prime of Type 2 or 3.
We have $2^8\mid L$.

By the minimality argument, we must have an $n$-pair with $n<6$ or a $6$-prime.
From Table \ref{table:P=3}, we find that there is no $n$-pair with $n<6$.
There are two $6$-primes: $193$ and $577$.

\subsubsection{Subcase 1: $q_2=193$}
Suppose that $q_2=193$ with $l_2=6$.
Then we have
\begin{align*}
m&=5\cdot 257 \cdot 13 \cdot 193 \prod_{i\geq 2} u_i=(2^9\cdot 3\cdot 2099+1)\prod_{i\geq 2} u_i.
\end{align*}

The product contains a Type 3 prime, so it is not empty.
By the minimality argument, there is either an $n$-pair $n<8$ or a $8$- or $9$-prime.
According to Table \ref{table:P=3}, there we only have the $8$-prime $769$.
Then
\begin{align*}
m&=5\cdot 257 \cdot 13 \cdot 193 \cdot 769 \prod_{i\geq 3} u_i
=(2^8\cdot 3\cdot 102\cdot 31963+1)\prod_{i\geq 3} u_i.
\end{align*}
As the product contains a Type 3 prime, it is not empty.
Then there must be an $n$-pair with $n<8$ or an $8$-prime by the minimality argument.
However, there is no such a pair or $8$-prime.
Hence $q_2\neq 193$.

\subsubsection{Subcase 2: $r_1=577$}
Suppose that $r_1=577$ with $s_1=6$.
Then we have
\begin{align*}
m&=5\cdot 257 \cdot 13 \cdot 577 \prod_{i\geq 2} u_i=(2^7\cdot 293+1)(2^8+1)\prod_{i\geq 2} u_i.
\end{align*}
It follows from the minimality arguments that we must have a $7$-prime as there is no $n$-pair $n<7$.
Thus we have $r_2=1153$ with $s_2=7$, and
\begin{align*}
m&=5\cdot 257 \cdot 13 \cdot 577 \cdot 1153 \prod_{i\geq 3} u_i\\
&=(2^{11}\cdot 3^2\cdot 602947+1)\prod_{i\geq 3} u_i.
\end{align*}

If the product $\prod_{i\geq 3} u_i=1$, then
\[\boxed{m=5\cdot 13 \cdot 257 \cdot 577 \cdot 1153}\]
is a Carmichael number since $L=2^8\cdot 3^2$ divides $m-1=2^{11}\cdot 3^2\cdot 602947$.
If the product is nonempty, then there is an $n$-prime with $n=7, 8, 9, 10, 11$.
From Table \ref{table:P=3}, there are two cases to consider: $8$-prime $769$ and $11$-prime $18433$.

\paragraph{Subsubcase 1: $q_2=769$.}
Suppose that $q_1=769$ with $l_2=8$.
Then
\begin{align*}
m&=5\cdot 257 \cdot 13 \cdot 577 \cdot 1153 \cdot 769 \prod_{i\geq 3} u_i\\
&=(2^{8}\cdot 3\cdot 2113 \cdot 5266441+1)\prod_{i\geq 3} u_i.
\end{align*}
The product must be empty otherwise the minimality argument require more $n$-pairs with $n < 8$ or $8$-primes, but we used all relevant primes from Table \ref{table:P=3}.
If the product is empty, then $m$ is not a Carmichael number as $3^2 \nmid m-1$.

\paragraph{Subsubcase 2: $r_3=18433$.}
Next, suppose that $r_3=18433$ with $s_3=11$.
Then
\begin{align*}
m&=5\cdot 257 \cdot 13 \cdot 577 \cdot 1153 \cdot 18433 \prod_{i\geq 3} u_i\\
&=(2^{13}\cdot 3^2\cdot 11 \cdot 29 \cdot 8710127+1)\prod_{i\geq 3} u_i.
\end{align*}
If the product is empty, then 
\[\boxed{m=5 \cdot 13 \cdot 257\cdot 577 \cdot 1153 \cdot 18433}\]
is a Carmichael number since $L=2^{11}\cdot 3^2$ divides $m-1=2^{13}\cdot 3^2\cdot 11 \cdot 29 \cdot 8710127$.

Suppose the product is nonempty. Since $2^{11}\mid L$, there is an $n$-prime with $n=11, 12, 13$.
This yields the $12$-prime $12289$ and we have
\begin{align*}
m&=5\cdot 257 \cdot 13 \cdot 577 \cdot 1153 \cdot 18433 \cdot 12289 \prod_{i\geq 4} u_i\\
&=(2^{12}\cdot 3 \cdot 13477 \cdot 15201615259 \cdot 8710127+1)\prod_{i\geq 4} u_i.
\end{align*}

As there is no $12$-prime and a possible pair, the product is empty.
But then $m$ is not a Carmichael number as $3^2\nmid m-1$.

In summary, there are exactly two Carmichael numbers
\[\boxed{5\cdot 13 \cdot 257 \cdot 577 \cdot 1153} \text{ and } \boxed{5 \cdot 13 \cdot 257 \cdot 577 \cdot 1153 \cdot 18433}\]
in Case 2.

\subsection{Case 3: Type 1 combination is $5\cdot 65537$}
Finally, we consider the case when the Type 1 combination is $5\cdot 65537$.
We have $q_1=2^2\cdot 3+1=13$.
Let
\[m=5\cdot 65537 \cdot 13 \prod_{i\geq 1}u_i=(2^6+1)(2^{16}+1)\prod_{i\geq 1}u_i,\]
where each $u_i$ is a prime of Type 2 or 3. We have $2^{16}\mid L$.

By the minimality argument, the product $\prod_{i\geq 1}u_i$ must contain either an $n$-pair with $n<5$ or a $6$-prime.
We see from Table \ref{table:P=3} that there is no such pair.
There are two $6$-primes: $193$ and $577$.
Now we consider these two cases individually.

\subsubsection{Subcase 1: $q_2=193$}
Let us consider the case $q_2=193$ with $l_2=6$.
Then we have
\begin{align*}
m&=5\cdot 65537 \cdot 13 \cdot 193 \prod_{i\geq 2}u_i=(2^8\cdot 7^2+1)(2^{16}+1)\prod_{i\geq 2}u_i.
\end{align*}
As there is no $n$-pair with $n<7$, we must have a $8$-prime: $q_3=769$ with $l_3=8$.
Then we have
\begin{align*}
m&=5\cdot 65537 \cdot 13 \cdot 193 \cdot 769 \prod_{i\geq 3}u_i=(2^{10}\cdot 9421+1)(2^{16}+1)\prod_{i\geq 3}u_i.
\end{align*}

As the product contains a Type 3 prime, it is nonempty.
Since there is no $n$-pair with $n<10$ and there is no $10$-prime for $\prod_{i\geq 3}u_i$, this case does not happen.

\subsubsection{Subcase 1: $r_1=577$}
Next, suppose that $r_1=577$ with $s_1=6$.
Then we have
\begin{align*}
m&=5\cdot 65537 \cdot 13 \cdot 577 \prod_{i\geq 2}u_i=(2^7\cdot 293+1)(2^{16}+1)\prod_{i\geq 2}u_i.
\end{align*}
Since $2^{16}\mid L$, the product $\prod_{i\geq 2}u_i\neq 1$, otherwise $7$ is the unique minimal power of $2$.
By the minimality argument with Table \ref{table:P=3} yields that there must be the $7$-prime $1153$.
Thus $r_2=1153$ with $s_2=7$, and we have
\begin{align*}
m&=5\cdot 65537 \cdot 13 \cdot 577 \cdot 1153 \prod_{i\geq 3}u_i\\
&=(2^8\cdot 31 \cdot 5449+1)(2^{16}+1)\prod_{i\geq 3}u_i.
\end{align*}

Again the minimality argument implies that there is a $8$-prime.
So $q_2=769$ with $l_2=8$, and we have
\begin{align*}
m&=5\cdot 65537 \cdot 13 \cdot 577 \cdot 1153 \cdot 769 \prod_{i\geq 4}u_i\\
&=(2^9\cdot 11 \cdot 29\cdot 97 \cdot 2099+1)(2^{16}+1)\prod_{i\geq 4}u_i.
\end{align*}

Since there is no $n$-pair with $n<8$ and there is no $9$-prime from Table \ref{table:P=3}, the product $\prod_{i\geq 4}u_i=1$. But then $2^9$ is the unique smallest power of $2$, which contradicts Theorem \ref{thm:power of 2}. Hence $m$ is not a Carmichael number in this case.

In conclusion, we obtain:
\begin{theorem}
	Let $m$ be a Carmichael number with $L=2^{\alpha}3^2$ for some $\alpha\in \N$.
	Assume that $m$ has at least one Fermat prime factor and $k_1=l_1\leq s_1$.
	Then $m$ is one of the following two Carmichael numbers:
	\[\boxed{5\cdot 13 \cdot 257 \cdot 577 \cdot 1153} \text{ and } \boxed{5\cdot 257 \cdot 13 \cdot 577 \cdot 1153 \cdot 18433}.\]
\end{theorem}

This exhausts all the possible $P=3, 5, 7, 11, 41, 43, 127, 19661$, and we complete the classification of Carmichael numbers for Case A ($k_1=l_1\leq s_1$).

\section{Case B: $k_1=s_1 < l_1$}
We next consider Case B.
We assume that $k_1=s_1< l_1$.

The first step is to reduce the number of possible $(P, k_1)$ from Table \ref{table:products of fermat primes}.
We remove those $(P, k_1)$ such that $2^{k_1}P^2+1$ is not prime.
Note that $s_1$ is even if $P\neq 3$ by Lemma \ref{lem:mod 3}.
Thus, we remove all $(P, 1)$ from Table \ref{table:products of fermat primes}.

We list the prime factorizations of $2^{k_1}P^2+1$ for $(P, k_1)$ with $k_1\geq 2$ in Table \ref{table:Case B many Candidates for P}.

\begin{table}[ht]
\caption{Prime factorization of $2^{k_1}P^2+1$}
\centering
\begin{tabular}{c |c c ||c c c}
\hline\hline
$(P, k_1)$ &   $2^{k_1}P^2+1$ & Prime?& $(P, k_1)$ &$2^{k_1}P^2+1$ & Prime?\\ [0.5ex] 
\hline
$(3, 2)$ & $37$ &prime  & $(3, 4)$ & $5\cdot 29$\\
$(7, 2)$ & $197$ & prime  &$(7,4)$& $5\cdot 157$\\
$(107, 2)$ & $41\cdot 1117$ & &$(13, 4)$ & $5\cdot 541$ \\
$(47,2)$ & $8837$ &prime  & $(2579, 4)$ & $1373 \cdot 77509$\\
$(83,2)$ & $17\cdot 1621$ && $(29, 4)$ & $13457$ & prime\\
$(43, 2)$ & $13 \cdot 569$ & & $(43, 4)$ & $5 \cdot 61 \cdot 97$\\
$(127, 2)$ & $149 \cdot 433$ && $(113, 4)$ & $5\cdot 29 \cdot 1409$\\
$(131, 2)$ & $5 \cdot 13729$ && $(127, 4)$ & $5\cdot 51613$\\
$(10631, 2)$ & $5 \cdot 197 \cdot 458957$ && $(3,8)$ & $5\cdot 461$\\
$(467, 2)$ & $60029 \cdot 135433$ && $(7,8)$& $5\cdot 13\cdot 193$\\
$(11, 2)$ & $5\cdot 97$ && $(241, 8)$ & $13 \cdot 1143749$\\
$(31, 2)$ & $5\cdot 769$ &&\\
$(151, 2)$ & $5\cdot 17 \cdot 29 \cdot 37$&&\\
$(331, 2)$ & $5\cdot 87649$&&\\[1ex]
\hline
\end{tabular}
\label{table:Case B many Candidates for P}
\end{table}

The only possible pairs are
\[(P, k_1)=(3, 2), (7, 2), (47, 2), (29, 4).\]
We list all the possible pairs $(P, k_1)$ together with combinations of Type 1 primes in Table \ref{table:Case B reduced candidates for P}.

\begin{table}[ht]
\caption{The possible pairs $(P, k_1)$ for Case B}
\centering
\begin{tabular}{|c| c|  }
\hline\hline
Combination of Type 1 primes & $(P, k_1)$  \\ [0.5ex] 
\hline
$5\cdot 17$ & $(3,2), (7,2)$\\
$5\cdot 257$ & $(3,2)$\\
$5\cdot 65537$ & $(3,2)$, $(7,2)$, $(47,2)$\\
$5\cdot 17 \cdot 257 \cdot 65537$ & $(3,2)$, $(7,2)$\\
$17\cdot 257 \cdot 65537$ & $(29, 4)$\\[1ex]
\hline
\end{tabular}

\label{table:Case B reduced candidates for P}
\end{table}

\section{The impossible case: $P=29$}
Let us consider the case $P=29$.
We see from Table \ref{table:Case B reduced candidates for P} that $17\cdot 257 \cdot 65537$ is the only pair for $P=29$.
Note that $2^{16} \mid L$.
We have $k_1=s_1=4$, and hence $r_1=2^4\cdot 29^2+1=13457$.
Let
\[m=17\cdot 257 \cdot 65537 \cdot 13457 \prod_{i \geq 1} u_i,\]
where $u_i$ is a prime of Type 2 or 3, or the product $\prod_{i \geq 1} u_i$ could be empty.
We have
\[m=(2^5\cdot 3\cdot 2383+1)(2^8+1)(2^{16}+1)\prod_{i \geq 1} u_i.\]

Since $2^{16}\mid L$, the product cannot be empty by the minimality argument.
It follows from $P=29\equiv 2 \pmod{3}$ that $l_i$ is odd and $s_i$ is even.
Since the parities of $l_i$ and $s_i$ are distinct, there is no $n$-pair for any $n$.
Since $2^{16} \mid L$, by the minimality argument, there is a $5$-prime (and it must be a Type 2 prime).
Thus we have $q_1=2^5\cdot 29+1=929$ with $l_1=5$.
Then we have
\begin{align*}
m&=17\cdot 257 \cdot 65537 \cdot 13457 \cdot 929\prod_{i \geq 2} u_i\\
&=(2^6\cdot 5^2\cdot 19 \cdot 6991)(2^8+1)(2^{16}+1)\prod_{i \geq 2} u_i
\end{align*}

Again since $2^{16}\mid L$, the product is not empty.
Since $6$ is even, the only possible $6$-prime is of Type 3. However, the number $2^6\cdot 29^2+1=5^2\cdot 2153$ is composite.
Thus we have:
\begin{theorem}
	If $P=29$, then there is no Carmichael number for Case B.
\end{theorem}

\section{The impossible case: $P=47$}
We now consider the case $P=47$.
From Table \ref{table:Case B reduced candidates for P}, the Type 1 combination is $5\cdot 65537$.
So $k_1=s_1=2$, and $r_1=2^2\cdot 47^2+1=8837$.
Let
\[m=5\cdot 65537 \cdot 8837 \prod_{i \geq 1} u_i=(2^3\cdot 3\cdot 7 \cdot 263+1)(2^{16}+1)\prod_{i \geq 1} u_i,\]
where $u_i$ is a Type 2 or 3 prime, or the product $\prod_{i \geq 1} u_i$ could be empty.

Since $2^{16}\mid L$, the product cannot be empty by the minimality argument.
Since $P=47\equiv 2 \pmod{3}$, we know that $l_i$ is odd and $s_i$ is even.
Hence there is no $n$-pair.
The only possible $3$-prime is of Type 2 but $2^3\cdot 47+1=13\cdot 29$ is composite.
This proves:
\begin{theorem}
	If $P=47$, then there is no Carmichael number for Case B.
\end{theorem}

\section{The impossible case: $P=7$}
We next deal with the case $P=7$ in Case B.
The combinations of Type 1 primes for $P=7$ are 
\[5\cdot 17, \qquad 5\cdot 65537, \qquad 5\cdot 17\cdot 257 \cdot 65537.\]
We have $k_1=s_1=2$, and thus $r_1=2^2\cdot 7^2+1=197$.
Note that $p_2=2^{k_2}+1$ with $k_2\geq 4$ in all cases.
We have $2^{4}\mid 2^{k_2}\mid L$.

Let
\[m=5\cdot 197 \cdot p_2 \prod_{i\geq 1}v_i=(2^3\cdot 3 \cdot 41+1)(2^{k_2}+1)\prod_{i\geq 1}v_i,\]
where $v_i$ is a prime of Type 1, 2, or 3, or the product could be empty.

Since $2^4\mid L$, there must be an $n$-pair with $n<3$ or $3$-prime otherwise $3$ is the unique $2$-power. 
Since $P=7\equiv 1 \pmod{3}$, we know $l_i, s_i$ are even.
Thus, there is no $3$-prime.
Since we already used 2-primes $p_1$, $r_1$ with $k_1=s_1=2$, there is no $2$-pair.
This proves:
\begin{theorem}
	If $P=7$, then there is no Carmichael number for Case B.
\end{theorem}

\section{Case: $P=3$ for Case B}
The last possible $P$ for Case B is $P=3$.
It follows from Table \ref{table:Case B reduced candidates for P} that the combinations of Type 1 primes for $P=3$ are
\[5\cdot 17, \qquad 5\cdot 257, \qquad 5\cdot 65537, \qquad 5 \cdot 17 \cdot 257 \cdot 65537.\]
In all cases, we have $k_1=s_1=2$, $k_2\geq 4$, and thus $r_1=2^2\cdot 3^2+1=37$ and $2^4\mid L$.

\subsection{Case 1: $17 \mid m$}
Let us fist consider the case when $17 \mid m$.
Let 
\[m=5\cdot 17 \cdot 37 \prod_{i \geq 1} v_i=(2^3\cdot 23+1)(2^4+1)\prod_{i \geq 1} v_i,\]
where $v_i$ is a prime of Type 1, 2, or 3. 
By the minimality argument, the product cannot be empty and there is an $n$-pair $n<4$ or a $3$-prime.
From Table \ref{table:P=3}, we see there is no such a pair and the only $3$-prime is $73$.
Then we have
\[m=5\cdot 17 \cdot 37 \cdot 73\prod_{i \geq 2} v_i=(2^4\cdot 3\cdot 4783+1)\prod_{i \geq 2} v_i.\]

Since there is no $4$-primes and possible $n$-pairs, the product must be empty. But if the product is empty, then $m$ is not a Carmichael number as $3^2$ does not divide $m-1$.

\subsection{Case 2: The Type 1 combination is $5\cdot 257$}
We next consider the case when the Type 1 combination is $5\cdot 257$.
Then
\[m=5\cdot 257 \cdot 37\prod_{i \geq 1} u_i=(2^3\cdot 3 \cdot 7 \cdot 283+1)\prod_{i \geq 1} u_i,\]
where $u_i$ is a prime of Type 2, or 3.
Note that $2^{8} \mid L$.
By the minimality argument, we have the $3$-prime $73$.
So
\[m=5\cdot 257 \cdot 37 \cdot 73\prod_{i \geq 2} u_i=(2^6\cdot 3 \cdot 18077 +1)\prod_{i \geq 2
} u_i,\]
If the product is empty, then $m$ is not a Carmichael number as $3^2$ does not divide $m-1$.
If the product is nonempty, then since there is no $n$-pair $n<6$, there must be a $6$-prime.
There are two $6$-primes from Table \ref{table:P=3}: $193$ and $577$.

\subsubsection{Subcase 1: $q_1=193$}
Suppose that $q_1=193$.
Then we have
\begin{align*}
m=5\cdot 257 \cdot 37 \cdot 73 \cdot 193\prod_{i \geq 3} u_i=(2^7\cdot 3^3 \cdot 13 \cdot 44729+1)\prod_{i \geq 3} u_i.
\end{align*}
If the product is empty, then $m$ is not a Carmichael number as $2^8\mid L$ but $2^8$ does not divide $m-1$.
By the minimality argument, there is a $7$-prime as no possible $n$-pairs.
The unique $7$-prime is $1153$. So we have
\begin{align*}
m&=5\cdot 257 \cdot 37 \cdot 73 \cdot 193 \cdot 1153\prod_{i \geq 4} u_i\\
&=(2^8\cdot 3^3 \cdot 11 \cdot 10158227+1)\prod_{i \geq 4} u_i.
\end{align*}

If the product is empty, then
\[ \boxed{m=5 \cdot 37 \cdot 73 \cdot 193 \cdot 257\cdot 1153}\]
is a Carmichael number since $L=2^8\cdot 3^2$ divides $m-1=2^8\cdot 3^3 \cdot 11 \cdot 10158227$.

If the product is nonempty, then there must be a $8$-prime, which is $769$. Then
\begin{align*}
m&=5\cdot 257 \cdot 37 \cdot 73 \cdot 193 \cdot 1153 \cdot 769 \prod_{i \geq 5} u_i=(2^9\cdot 3 \cdot 19 \cdot 20351473151+1)\prod_{i \geq 5} u_i.
\end{align*}
As there is no $9$-primes or possible $n$-pairs, the product must be empty. Then $m$ is not a Carmichael number as $3^2\nmid m-1$.

\subsubsection{Subcase 2: $r_3=577$}
Suppose now that $r_3=577$. Then
\begin{align*}
m=5\cdot 257 \cdot 37 \cdot 73 \cdot 577\prod_{i \geq 3} u_i=(2^{11}\cdot 3 \cdot 325951 +1)\prod_{i \geq 3} u_i.
\end{align*}
If the product is empty, then $m$ is not a Carmichael number as $3^2\nmid m-1$.
If the product is nonempty, then by the minimality argument, there is an $n$-pair with $n<8$ or an $n$-prime with $n=8, 9, 10, 11$.
There is no possible pairs. We have the $8$-prime $769$ and the $11$-prime $18433$.

\paragraph{Subsubcase 1: $8$-prime $769$} Suppose $769$ is a divisor of $m$. Then
\begin{align*}
m=5\cdot 257 \cdot 37 \cdot 73 \cdot 577 \cdot 769\prod_{i \geq 4} u_i=(2^{8}\cdot 3^4 \cdot 283 \cdot 262433+1)\prod_{i \geq 4} u_i.
\end{align*}
If the product is empty, then
\[\boxed{m=5 \cdot 37 \cdot 73\cdot 257 \cdot 577 \cdot 769}\]
is a Carmichael number since $L=2^8\cdot 3^2$ divides $m-1=2^{8}\cdot 3^4 \cdot 283 \cdot 262433$.

\paragraph{Subsubcase 2: $11$-prime $18433$}
Next, suppose that $18433$ is a divisor of $m$.
Then
\begin{align*}
m=5\cdot 257 \cdot 37 \cdot 73 \cdot 577 \cdot 18433\prod_{i \geq 4} u_i=(2^{12}\cdot 3 \cdot3004127393 \cdot +1)\prod_{i \geq 4} u_i.
\end{align*}
If the product is empty, then $m$ is not a Carmichael number as $3^2\nmid m-1$.
If the product is nonempty, then the minimality argument yields that there is a $n$-prime with $n=8, 9, 10, 11, 12$.
From Table \ref{table:P=3}, we have the $8$-prime $769$ and the $12$-prime $12289$.

\subparagraph{Subsubsubcase 2-1: $8$-prime $769$}
Suppose that $769\mid m$. Then
\begin{align*}
m&=5\cdot 257 \cdot 37 \cdot 73 \cdot 577 \cdot 18433 \cdot 769\prod_{i \geq 5} u_i\\
&=(2^{8}\cdot 3^2 \cdot1667 \cdot 7391078473 +1)\prod_{i \geq 5} u_i.
\end{align*}
If the product is empty, then $m$ is not a Carmichael number as
$2^{11}\mid L$ but $2^{11} \nmid m-1$.
But since $2^{11}\mid L$ and there is no $8$-prime or possible pairs, the product must be empty.

\subparagraph{Subsubsubcase 2-2}: $12$-prime $12289$.
Next, we consider the case $12289\mid m$.
Then we have
\begin{align*}
m&=5\cdot 257 \cdot 37 \cdot 73 \cdot 577 \cdot 18433 \cdot 12289\prod_{i \geq 5} u_i\\
&=(2^{13}\cdot 3^2 \cdot 2357 \cdot 2610502159 +1)\prod_{i \geq 5} u_i.
\end{align*}
If the product is empty, then 
\[\boxed{m=5\cdot 257 \cdot 37 \cdot 73 \cdot 577  \cdot 12289 \cdot 18433}\]
is a Carmichael number since $L=2^{12}\cdot 3^2$ divides \[m-1=2^{13}\cdot 3^2 \cdot 2357 \cdot 2610502159.\]
Suppose the product is nonempty. Then since $2^{12}\mid L$, by the minimality argument, there must be an $n$-pair with $n<11$ or $12$-prime or $13$-prime.
Table \ref{table:P=3} shows that this is not the case.

\subsection{Case 3: The Type 1 combination is $5\cdot 65537$}
The last case for $P=3$ in Case B is when the Type 1 combination is $5\cdot 65537$.
Then
\[m=5\cdot 65537 \cdot 37\prod_{i \geq 1} u_i=(2^3\cdot 23 +1)(2^{16}+1)\prod_{i \geq 1} u_i,\]
where $u_i$ is a prime of Type 2 or 3.
Note that $2^{16}\mid L$.
By the minimality argument, the product cannot be empty, and there must be a $3$-prime.
The only $3$-prime is $73$. Then
\[m=5\cdot 65537 \cdot 37 \cdot 73\prod_{i \geq 2} u_i=(2^6\cdot 211 +1)(2^{16}+1)\prod_{i \geq 2} u_i.\]
Again, the product cannot be empty, and there must be a $6$-prime.
There are two $6$-primes: $193$ and $577$.

\subsubsection{Subcase 1: $6$-prime is $193$}
Suppose that $193\mid m$.
Then we have
\[m=5\cdot 65537 \cdot 37 \cdot 73 \cdot 193\prod_{i \geq 3} u_i=(2^7\cdot 7 \cdot 2909 +1)(2^{16}+1)\prod_{i \geq 3} u_i.\]
Furthermore, the product cannot be empty, and there must be a $7$-prime.
The only $7$-prime is $1153$ and we have
\[m=5\cdot 65537 \cdot 37 \cdot 73 \cdot 193 \cdot 1153\prod_{i \geq 3} u_i=(2^9\cdot 641 \cdot 9157 +1)(2^{16}+1)\prod_{i \geq 3} u_i.\]
And again, by the minimality argument the product is not empty, but there is no $9$-prime or admissible pairs.

\subsubsection{Subcase 1: $6$-prime is $577$}
Suppose next that $577\mid m$.
Then we have
\[m=5\cdot 65537 \cdot 37 \cdot 73 \cdot 577\prod_{i \geq 3} u_i=(2^8\cdot 61 \cdot 499 +1)(2^{16}+1)\prod_{i \geq 3} u_i.\]
By the minimality argument, the product is not empty and the $8$-prime $769$ divides $m$.
Then
\[m=5\cdot 65537 \cdot 37 \cdot 73 \cdot 577 \cdot  769\prod_{i \geq 3} u_i=(2^9\cdot 7^2 \cdot 238853 +1)(2^{16}+1)\prod_{i \geq 3} u_i.\]
The product cannot be empty otherwise $2^9$ is the unique smallest power of $2$. However, there is no $9$-prime or admissible pairs so the product must be empty. 
Hence there is no Carmichael number in this case.

In summary, we obtain:
\begin{theorem}
	Let $m$ be a Carmichael number with $L=2^{\alpha}\cdot 3^2$. In Case B ($k_1=s_1 <l_1)$, the following three are all Carmichael numbers of this type:
	\begin{align*}
m&=5 \cdot 37 \cdot 73 \cdot 193 \cdot 257\cdot 1153\\
m&=5 \cdot 37 \cdot 73\cdot 257 \cdot 577 \cdot 769\\
m&=5 \cdot 37 \cdot 73 \cdot 257 \cdot 577  \cdot 12289 \cdot 18433
\end{align*}
	
\end{theorem}

\section{Case C: $l_1=s_1<k_1$}
Let $m$ be a Carmichael number with $L=2^{\alpha}P^2$ as before.
In this section, we consider Case C: $l_1=s_1<k_1$.

Note that $1\leq l_1 <k_2$. From Table \ref{table:products of fermat primes}, the possible values for $k_1$ are $2, 4, 8$.
\section{Case 1: $k_1=2$}
Consider the case $k_1=2$. Then we must have $l_1=s_1=1$.
By Lemma \ref{lem:mod 3}, this yields that $P=3$ as $s_1$ is odd.
Then we have $q_1=2\cdot 3+1=7$ and $r_1=2\cdot 3^2+1=19$.
From Table \ref{table:products of fermat primes}, the Type 1 combinations for $P=3$ with $k_1=2$ are
\[5 \cdot 17, \qquad 5\cdot 257, \qquad 5\cdot 65537, \qquad 5\cdot 17 \cdot 257 \cdot  65537. \]
In all cases, we have $2^{4}\mid L$.
Let
\[m=5\cdot 7 \cdot 19 \prod_{i \geq 1}v_i=(2^3\cdot 83+1)\prod_{i \geq 1}v_i,\]
where $v_i$ is a prime of Type 1, 2, or 3, or it could be empty.
As $2^4\mid L$, by the minimality argument (Theorem \ref{thm:minimality}) there must be a $2$-pair of a $3$-prime in $\prod_{i \geq 1}v_i$.
We see from Table \ref{table:P=3} that $(13, 37)$ is the only $2$-pair and $73$ is the only $3$-prime.

\subsection{Subcase 1: the $2$-pair $(13, 37)$}
Suppose that $m$ has the $2$-pair $(13, 37)$.
Then
\[m=5\cdot 7 \cdot 19 \cdot 13 \cdot 37\prod_{i \geq 3}v_i=(2^3\cdot 39983+1)\prod_{i \geq 3}v_i,\]
By the minimality argument, $m$ has the $3$-prime $73$.
Then
\[m=5\cdot 7 \cdot 19 \cdot 13 \cdot 37 \cdot 73\prod_{i \geq 3}v_i=(2^7\cdot 182423+1)\prod_{i \geq 3}v_i,\]
Note that we have used all possible primes with $2$-power less than $4$ (See Table \ref{table:P=3}) and $17$ is the unique $4$-prime.
Thus $17 \nmid m$, otherwise $2^4$ is the unique smallest power of $2$.
It follows that $2^{8}\mid L$.
Hence by the minimality argument, there is either an $n$-pair with $n<7$ or a $7$-prime.
Thus, there are two cases: the $6$-pair $(193, 577)$ and the $7$-prime $1153$.

\subsubsection{Subsubcase 1: the $6$-pair $(193, 577)$}
Suppose that $m$ has the $6$-pair $(193, 577)$.
Then we have
\begin{align*}
m=5\cdot 7 \cdot 19 \cdot 13 \cdot 37 \cdot 73 \cdot 193 \cdot 577\prod_{i \geq 5}v_i=(2^7\cdot 2003 \cdot 10142191+1)\prod_{i \geq 5}v_i.
\end{align*}
As $2^8\mid L$, the minimality arguments yields that $m$ has the $7$-prime $1153$.
Then
\begin{align*}
m&=5\cdot 7 \cdot 19 \cdot 13 \cdot 37 \cdot 73 \cdot 193 \cdot 577 \cdot 1153\prod_{i \geq 6}v_i\\
&=(2^8\cdot 1092373 \cdot 10721143+1)\prod_{i \geq 6}v_i.
\end{align*}
If the product $\prod_{i \geq 6}v_i=1$, then $m$ is not a Carmichael number because $3^2\nmid m-1$.
Thus the product is nonempty. As there is no more admissible pairs, there must be a $8$-prime.
There are two $8$-primes: $257$ and $769$

\paragraph{Subsubsubcase 1: the $8$-prime: $257$}
Suppose that $257\mid m$.
Then we have
\begin{align*}
m&=5\cdot 7 \cdot 19 \cdot 13 \cdot 37 \cdot 73 \cdot 193 \cdot 577 \cdot 1153 \cdot 257 \prod_{i \geq 7}v_i\\
&=(2^{10}\cdot 3^2 \cdot 83607005432809+1)\prod_{i \geq 7}v_i.
\end{align*}
If the product is empty, then
\[\boxed{m=5\cdot 7  \cdot 13 \cdot 19 \cdot 37 \cdot 73 \cdot 193 \cdot 257 \cdot 577 \cdot 1153 }\]
is a Carmichael number since $L=2^8\cdot 3^2$ divides $m-1=2^{10}\cdot 3^2 \cdot 83607005432809$.

Note that $2^8\mid L$. If the product is nonempty, then there is an $n$-prime with $n=8, 9, 10, 11$ as there is no admissible pairs.
So $m$ has the $8$-prime $769$, and then
\begin{align*}
m&=5\cdot 7 \cdot 19 \cdot 13 \cdot 37 \cdot 73 \cdot 193 \cdot 577 \cdot 1153 \cdot 257 \cdot 769 \prod_{i \geq 7}v_i\\
&=(2^{8}\cdot 3 \cdot 11 \cdot 31 \cdot 2262537965202233+1)\prod_{i \geq 7}v_i.
\end{align*}
Since there is no more admissible pair or $8$-prime, the product is empty, but then $m$ is not a Carmichael number as $3^2\nmid m-1$.

\paragraph{Subsubsubcase 2: the $8$-prime: $769$}
Next, suppose that $769\mid m$.
Then
\begin{align*}
m&=5\cdot 7 \cdot 19 \cdot 13 \cdot 37 \cdot 73 \cdot 193 \cdot 577 \cdot 1153 \cdot 769 \prod_{i \geq 7}v_i\\
&=(2^{9}\cdot 4503066806229347+1)\prod_{i \geq 7}v_i.
\end{align*}
If the product is empty, then $m$ is not a Carmichael number as $3\nmid m-1$.
Then by the minimality argument, there is a $8$- or $9$-prime.
Since there is no $9$-prime, $m$ must have the $8$-prime $257$.
Then this case reduces to the previous case.

\subsubsection{Subsubcase 2: the $7$-prime $1153$}
Now suppose that $r_4=1153\mid m$.
Then
\begin{align*}
m=5\cdot 7 \cdot 19 \cdot 13 \cdot 37 \cdot 73 \cdot 1153
 \cdot 577\prod_{i \geq 5}v_i=(2^{12}\cdot 11 \cdot 597539 +1)\prod_{i \geq 5}v_i.
\end{align*}
Since $2^8\mid L$, by the minimality argument, there is an $n$-prime with $8\leq n \leq 12$.
Those are: $8$-primes $769$, $257$, the $11$-prime $18433$, the $12$-prime $12289$ from Table \ref{table:P=3}.

\paragraph{Subsubsubcase 1: the $8$-prime $769$}
When $769\mid m$, we have
\begin{align*}
m=5\cdot 7 \cdot 19 \cdot 13 \cdot 37 \cdot 73 \cdot 1153
 \cdot 577 \cdot 769 \prod_{i \geq 6}v_i=(2^{8}\cdot 89 \cdot 1483\cdot 612737 +1)\prod_{i \geq 6}v_i.
\end{align*}
If the product is empty, then $m$ is not a Carmichael number as $3\nmid m-1$.
Thus it has the $8$-prime $257$, and then
\begin{align*}
m&=5\cdot 7 \cdot 19 \cdot 13 \cdot 37 \cdot 73 \cdot 1153
 \cdot 577 \cdot 769 \cdot 257\prod_{i \geq 7}v_i\\
 &=(2^{10}\cdot 3^2 \cdot 61\cdot 9464682529 +1)\prod_{i \geq 6}v_i.
\end{align*}
The product must be empty since there is no $n$-prime for $n=8, 9, 10$.
Then
\[\boxed{m=5\cdot 7  \cdot 13 \cdot 19\cdot 37 \cdot 73 \cdot 257
 \cdot 577 \cdot 769 \cdot 1153}\]
 is a Carmichael number since $L=2^8 \cdot 3^2$ divides $m-1$.

\paragraph{Subsubsubcase 2: the $8$-prime $257$}
When $257\mid m$, then
\begin{align*}
m=5\cdot 7 \cdot 19 \cdot 13 \cdot 37 \cdot 73 \cdot 1153
 \cdot 577 \cdot 257 \prod_{i \geq 6}v_i=(2^{8}\cdot 3 \cdot 43 \cdot 209518481 +1)\prod_{i \geq 6}v_i.
\end{align*}
If the product is empty, then $m$ is not a Carmichael number as $3^2\nmid m-1$.
Thus $m$ has the $8$-prime $769$, and this case reduces to the previous case.

\paragraph{Subsubsubcase 3: the $11$-prime $18433$}
If $18433\mid m$, then
\begin{align*}
m=5\cdot 7 \cdot 19 \cdot 13 \cdot 37 \cdot 73 \cdot 1153
 \cdot 577 \cdot 257 \cdot 18433\prod_{i \geq 6}v_i=(2^{11} x +1)\prod_{i \geq 6}v_i,
\end{align*}
where $x$ is an odd number.
(From now on, to save space, I use $x$ for an odd number, especially when $m$ is not a Carmichael number.)
Note that $2^{11}\mid L$ but there is no more $11$-prime.
Thus the product must be empty.
If the product is empty, then $m$ is not a Carmichael number since $3^2\nmid m-1$.

\paragraph{Subsubsubcase 4: $12$-prime $12289$}
When $12289\mid m$, we have
\begin{align*}
m=5\cdot 7 \cdot 19 \cdot 13 \cdot 37 \cdot 73 \cdot 1153
 \cdot 577 \cdot 257 \cdot 12289\prod_{i \geq 6}v_i=(2^{14} x +1)\prod_{i \geq 6}v_i.
\end{align*}
Since $2^{12}\mid L$, there is an $n$-prime with $n=12, 13, 14$.
From Table \ref{table:P=3}, such a prime must be the  $14$-prime $147457$.
Then
\begin{align*}
m=5\cdot 7 \cdot 19 \cdot 13 \cdot 37 \cdot 73 \cdot 1153
 \cdot 577 \cdot 257 \cdot 12289 \cdot 147457 \prod_{i \geq 7}v_i=(2^{15} x +1)\prod_{i \geq 7}v_i
\end{align*}
and $2^{14} \mid L$.
Since there is no $14$-prime and $15$-prime, the product must be empty. But then $m$ is not a Carmichael number as $3\nmid m-1$.

\subsection{Subcase 2: the $3$-prime $73$}
Next, we consider the case $r_2=73$ with $s_2=3$.
Then
\[m=5\cdot 7 \cdot 19 \cdot 73 \prod_{i \geq 2}v_i=(2^5\cdot 37\cdot 41+1)\prod_{i \geq 2}v_i,\]
Since $2^4\mid L$, by the minimality argument, there is an $n$-prime with $n=4, 5$.
From Table \ref{table:P=3}, there are the $4$-prime $17$ and $5$-prime $97$.

\subsubsection{Subsubcase 1: the $4$-prime $17$}
When $17\mid m$, we have
\[m=5\cdot 7 \cdot 19 \cdot 73 \cdot 17 \prod_{i \geq 3}v_i=(2^4\cdot 3^2\cdot 11 \cdot 521+1)\prod_{i \geq 3}v_i,\]
Since there is no more $4$-prime, the product is empty.
Then
\[\boxed{m=5\cdot 7 \cdot 17 \cdot 19   \cdot 73}\]
is a Carmichael number since $L=2^4\cdot 3^2$ divides $m-1$.

\subsubsection{Subsubcase 2: the $5$-prime $97$}
When $97\mid m$, we have
\[m=5\cdot 7 \cdot 19 \cdot 73 \cdot 97 \prod_{i \geq 3}v_i=(2^9\cdot 17 \cdot 541 +1)\prod_{i \geq 3}v_i.\]
If the product is empty, then $m$ is not a Carmichael number as $3\nmid m-1$.
Since $2^{5}\mid L$, by the minimality argument, there is an $n$-prime with $n=5, 6, 7, 8$.
From Table \ref{table:P=3}, there are five possible such prime numbers: $6$-primes $193$, $577$, the $7$-prime $1153$, and $8$-primes $769$, $257$.
We consider these cases individually below.

\textbf{Convention}: We use $x$ for an odd number such that $2^ax+1$ is not a Carmichael number.

\paragraph{Subsubsubcase 1: $6$-primes $193$}\label{Case C, k=2, subsubsubcase 1 6-prime 193}
When $193\mid m$, then
\[m=5\cdot 7 \cdot 19 \cdot 73 \cdot 97 \cdot 193 \prod_{i \geq 4}v_i=(2^6 x +1)\prod_{i \geq 4}v_i.\]
Our convention is that $x$ is an odd number and $2^6x+1$ is not a Carmichael number.
By the minimality argument, there must be the $6$-prime $577$.
Then
\[m=5\cdot 7 \cdot 19 \cdot 73 \cdot 97 \cdot 193 \cdot 577 \prod_{i \geq 5}v_i=(2^8 x +1)\prod_{i \geq 5}v_i.\]
Since $2^6\mid L$, the possible values for $v_6$ are the $7$-prime $1153$ and $8$-primes $769$, $257$.

\subparagraph{When $1153\mid m$}, we have 
\[m=5\cdot 7 \cdot 19 \cdot 73 \cdot 97 \cdot 193 \cdot 577 \cdot 1152\prod_{i \geq 6}v_i=(2^7 x +1)\prod_{i \geq 6}v_i.\]
There is no more $7$-prime. No Carmichael number exists in this case.

\subparagraph{When $769\mid m$}, we have 
\[m=5\cdot 7 \cdot 19 \cdot 73 \cdot 97 \cdot 193 \cdot 577 \cdot 769\prod_{i \geq 6}v_i=(2^{12}x +1)\prod_{i \geq 6}v_i.\]
Note that $2^8\mid L$. There is an $n$-prime with $n=8, 9, 10, 11, 12$. 
From Table \ref{table:P=3}, these are the $8$-prime $257$, the $11$-prime $18433$, and the $12$-prime $12289$.
\begin{enumerate}
	\item When $257\mid m$, we have
	\[m=5\cdot 7 \cdot 19 \cdot 73 \cdot 97 \cdot 193 \cdot 577 \cdot 769 \cdot 257 \prod_{i \geq 7}v_i=(2^{8}x +1)\prod_{i \geq 7}v_i\]
	and there is no further $8$-prime. No Carmichael number exists in this case.

\item When $18433\mid m$, we have
	\[m=5\cdot 7 \cdot 19 \cdot 73 \cdot 97 \cdot 193 \cdot 577 \cdot 769 \cdot 18433 \prod_{i \geq 7}v_i=(2^{11} x +1)\prod_{i \geq 7}v_i.\]
	Note $2^{11} \mid L$ and there is no $11$-prime. No Carmichael number exists in this case.
\item When $12289\mid m$.
\item 	\[m=5\cdot 7 \cdot 19 \cdot 73 \cdot 97 \cdot 193 \cdot 577 \cdot 769 \cdot 12289 \prod_{i \geq 7}v_i=(2^{14} x +1)\prod_{i \geq 7}v_i.\]
Note that $2^{12}\mid L$. So there is an $n$-prime with $n=12, 13, 14$.
Table \ref{table:P=3} gives the $14$-prime $147457$.
Then combining this number yields
\[m=(2^{16}+1)\prod_{i \geq 8}v_i.\]
Since $2^{14}\mid L$, there is an $n$-prime with $n=14, 15, 16$.
Table \ref{table:P=3} gives the $16$-prime $65537$.
Then including this factor, we have
\[m=(2^{17}+1)\prod_{i \geq 9}v_i.\]
Since $2^{16}\mid L$, we must have the $17$-prime $1179649$.
This gives
\[m=(2^{18}+1)\prod_{i \geq 10}v_i.\]
Furthermore, $m$ is divisible by the $18$-prime $786433$.
Then $m$ is 
\begin{align*}
&5\cdot 7 \cdot 19 \cdot 73 \cdot 97 \cdot 193 \cdot 577 \cdot 769 \cdot 12289 \cdot 147457 \cdot 65537 \cdot 1179649 \cdot 786433\prod_{i \geq 11}v_i\\
&=(2^{23}\cdot 3^3 \cdot 13 \cdot 14783 \cdot 1020702401040725135124085171+1)\prod_{i \geq 11}v_i
\end{align*}
Since there is no more $n$-prime for $18\leq n \leq 23$, the product must be empty.
Then
\[\boxed{5\cdot 7 \cdot 19 \cdot 73 \cdot 97 \cdot 193 \cdot 577 \cdot 769 \cdot 12289 \cdot 147457 \cdot 65537 \cdot 1179649 \cdot 786433}\]
is a Carmichael number since $L=2^{18}\cdot 3^2$ divides $m-1$.
\end{enumerate}

\subparagraph{When $257\mid m$}, we have 
\[m=5\cdot 7 \cdot 19 \cdot 73 \cdot 97 \cdot 193 \cdot 577 \cdot  257 \prod_{i \geq 6}v_i=(2^{9}x +1)\prod_{i \geq 6}v_i.\]
Since $2^8\mid L$, there is an $n$-prime with $n=8, 9$.
As there is no $9$-prime in Table \ref{table:P=3}, we have the $8$-prime $769$.
Combining $769$, we have
\[m=(2^{8}x +1)\prod_{i \geq 7}v_i.\]
As there is no more $8$-prime, no Carmichael number exists in this case.

\paragraph{Subsubsubcase 2: the $6$-prime $577$}
When $577\mid m$, then
\[m=5\cdot 7 \cdot 19 \cdot 73 \cdot 97 \cdot 577 \prod_{i \geq 4}v_i=(2^6 x +1)\prod_{i \geq 4}v_i.\]
Since $2^6\mid L$, $m$ is divisible by the $6$-prime $193$.
Then this reduces to the case dealt in Section \ref{Case C, k=2, subsubsubcase 1 6-prime 193}.

\paragraph{Subsubsubcase 3: the $7$-prime $1153$}
When $1153\mid m$, we have 
\[m=(2^7x+1)\prod_{i \geq 4}v_i.\]
Since $2^7\mid L$ but there is no $7$-prime, no Carmichael number exists in this case.

\paragraph{Subsubsubcase 4: the $8$-prime $769$}
When $769\mid m$, we have 
\[m=(2^8x+1)\prod_{i \geq 4}v_i.\]
Since $2^8\mid L$, the $8$-prime $257$ divides $m$.
Then including $257$, we have
\[m=(2^9x+1)\prod_{i \geq 5}v_i.\]
As there is no more $8$-prime and $9$-prime, no Carmichael number exists in this case.

\paragraph{Subsubsubcase 5: the $8$-prime $257$}
When $257\mid m$, we have
\[m=5\cdot 7 \cdot 19 \cdot 73 \cdot 97 \cdot 257 \prod_{i \geq 4}v_i=(2^8\cdot 3^2 \cdot 23 \cdot 41 \cdot 557+1)\prod_{i \geq 4}v_i.\]
If the product is empty, then
\[\boxed{m=5\cdot 7 \cdot 19 \cdot 73 \cdot 97 \cdot 257}\]
is a Carmichael number since $L=2^8\cdot 3^2$ divides $m-1$.

Otherwise, since $2^8\mid L$, the $8$-prime $769$ divides $m$, and
\[m=(2^9x+1)\prod_{i \geq 5}v_i.\]
As there is no more $8$-prime and $9$-prime, no other Carmichael number exists in this case.

 \section{Case 2: $k_1=4$}
 We consider the case $k_1=4$ in Case C.
 We have three cases to consider $l_1=s_1=1, 2, 3$.
 
 \subsection{$l_1=s_1=1$}
 Suppose that $l_1=s_1=1$. 
 Since $s_1$ is odd, we must have $P=3$ by Lemma \ref{lem:mod 3}. 
 So $q_1=7$ and $r_1=19$.
 
 There are two combinations of Type 1 primes: $17\cdot 257$ and $17\cdot 65537$ from Table \ref{table:products of fermat primes}.
 In either case, we have $2^{4}\mid L$.

 Let 
 \[m=17\cdot (2^{k_2}+1) \cdot 7 \cdot 19 \prod_{i \geq 1} u_i=(2^2 \cdot 5 \cdot 113+1)(2^{k_2}+1)\prod_{i\geq 1}u_i,\]
 where $u_i$ is a prime of Type 2 or 3, or the product could be empty, and $k_2$ is either $4$ or $16$.

 By the minimality argument (Theorem \ref{thm:minimality}), there must be a $2$-prime: $q_1=13$ or $r_2=37$.
 
 \subsubsection{Subcase 1: the $2$-prime $13$}
 When $13\mid m$, we have
 \begin{align*}
m&=17\cdot (2^{k_2}+1) \cdot 7 \cdot 19 \cdot 13 \prod_{i \geq 2} u_i\\
 &=(2^4 \cdot 11 \cdot 167+1)(2^{k_2}+1)\prod_{i\geq 2}u_i\\
 &=(2^4x+1)\prod_{i\geq 2}u_i
\end{align*}
for both values of $k_2$.
Here we use the same convention of $x$ as before: $x$ is an odd number and $2^4x+1$ is not a Carmichael number.
Since $2^4\mid L$ and there is no more $4$-prime and $n$-pair with $n<4$, no Carmichael number exists in this case.

\subsubsection{Subcase 2: the $2$-prime $37$}
Suppose that $37\mid m$. Then
 \begin{align*}
m&=17\cdot (2^{k_2}+1) \cdot 7 \cdot 19 \cdot 37 \prod_{i \geq 2} u_i=(2^3 \cdot 10457+1)(2^{k_2}+1)\prod_{i\geq 2}u_i\\
\end{align*}
By the minimality argument, there is a $3$-prime: $73$.
So we have
\begin{align*}
m&=17\cdot (2^{k_2}+1) \cdot 7 \cdot 19 \cdot 37 \cdot 73 \prod_{i \geq 2} u_i=(2^4 x+1)\prod_{i\geq 2}u_i
\end{align*}
for both values of $k_2$.
Since there is no $4$-prime, no Carmichael number exists in this case.

\subsection{$l_1=s_1=2$}
We now consider the case $l_1=s_1=2$ in Case C.
As $l_1$ is even, this implies that either $P=3$ or $P\equiv 1 \pmod{3}$ by Lemma \ref{lem:mod 3}.

\subsubsection{Case: $P=3$}
Suppose that $P=3$. Then we have $q_1=13$ and $r_1=37$.
Then Type 1 prime combinations are either $17\cdot 257$ or $17\cdot 65537$ from Table \ref{table:products of fermat primes}.
In either case, we have $2^4\mid L$.
Let
\[m=17\cdot (2^{k_2}+1)\cdot 13\cdot 37 \prod_{i\geq 1}u_i,\]
where $u_i$ is a prime of Type 2 or 3, or the product could be empty. Here $k_2$ is either $4$ or $16$.
Then we compute
\[m=(2^4x+1) \prod_{i\geq 1}u_i\]
for both values of $k_2$.
Here $x$ is an odd number and $2^4x+1$ is not a Carmichael number.
Since $2^4\mid L$, the minimality argument (Theorem \ref{thm:minimality}) implies that the product has either $n$-pair with $n<4$ or $4$-prime.
Table \ref{table:P=3} shows that this is not the case.

\subsubsection{Case: $P\equiv 1 \pmod{3}$}
Let $P\equiv 1 \pmod{3}$.
From Table \ref{table:products of fermat primes}, we find primes $P\equiv 1 \pmod{3}$ with $k_1=4$ such that $2^2P+1$ and $2^2P^2+1$ are both prime.
See Table \ref{table:P=1 (3) possible P}.
\begin{table}[ht]
\caption{Possible $P\equiv 1 \pmod{3}$ with $l_1=s_1=2$, $k_1=4$}
\centering
\begin{tabular}{|c| c| c |c |}
\hline\hline
$P$ & $2^2\cdot P+1$ & $2^2\cdot P^2+1$ & Type 1 prime combination \\ [0.5ex] 
\hline
$7$ & $29$ prime & $197$ prime   & $17\cdot 257$ \\
$13$ & $53$ prime & $677$ prime  & $17 \cdot 257$ \\
$43$ & $173$ prime& $\times$ & \\
$127$ &  $509$ prime& $\times$ &   \\
$2579$ &  $\times$ & $\times$ & \\
[1ex]
\hline
\end{tabular}

$\times$=composite
\label{table:P=1 (3) possible P}
\end{table}

Thus, we consider $P=7$ and $P=13$.
\paragraph{Subcase 1: $P=7$}
When $P=7$, we have $q_1=29$ and $r_1=197$.
The only Type 1 combination for $P=7$ is $17\cdot 257$.

Let
\[m=17\cdot 257 \cdot 29 \cdot 197 \prod_{i \geq 1}u_i=(2^5\cdot 3^3\cdot 7 \cdot 4127+1)\prod_{i \geq 1}u_i,\]
where $u_i$ is a Type 2 or 3 prime, or the product could be empty.
Note that $2^8\mid L$.

By the minimality argument, there must be an $n$-pair with $n<5$ or a $5$-prime.
However this is impossible from Table \ref{table:P=7}.
(Recall that $l_i, s_i$ are even as $P\equiv 1 \pmod{3}$.) No Carmichael number exists in this case.

\paragraph{Subcase 1: $P=13$}
When $P=13$, we have $q_1=53,$ and $r_1=677$. The only Type 1 combination is $17 \cdot 257$.
Note that $2^8\mid L$.
Let
\[m=17\cdot 257 \cdot 53 \cdot 677 \prod_{i \geq 1}u_i=(2^3\cdot 293 \cdot 79813+1)\prod_{i \geq 1}u_i.\]
where $u_i$ is a Type 2 or 3 prime, or the product could be empty.
By the minimality argument, there must be $3$-prime but since $l_i, s_i$ are even as $13\equiv 1 \pmod{3}$, there is no $3$-prime. No Carmichael number exists in this case.

 \subsection{$l_1=s_1=3$}
 If $l_1=s_1=3$, then $P=3$ since $s_1$ is odd by Lemma \ref{lem:mod 3}.
 However, $2^3\cdot 3+1=5^2$ is not a prime. Hence No Carmichael number exists in this case.

\section{Case 3: $k_1=8$}
We consider the case $l_1=s_1 <k_1=8$.
There are seven cases according to the value of $l_1=1, 2, \dots, 7$. Note that $257 \cdot 65537$ is the unique Type 1 combination with $k_1=8$.

\subsubsection{Subcase 1: $l_1=s_1=1$}
When $l_1=s_1=1$, as $s_1$ is odd, we must have $P=3$.
Then $q_1=7$ and $r_1=19$.
Let 
\[m= 257 \cdot 65537 \cdot 7 \cdot 19 \prod_{i \geq 1}u_i=(2^2\cdot 3\cdot 11+1)\cdot 257 \cdot 65537 \cdot  \prod_{i \geq 1}u_i,\]
where $u_i$ is a Type 2 or 3 prime, or the product could be empty.
Note that $2^{16}\mid L$.

By the minimality argument (Theorem \ref{thm:minimality}) we must have a $2$-prime.
Then we have either $q_2=13$ or $r_2=37$. (See Table \ref{table:P=3}.)

\paragraph{Subsubcase 1: $q_2=13$}
When $13\mid m$, then we have
\[m=257 \cdot 65537 \cdot 7 \cdot 19 \cdot 13 \prod_{i \geq 2}u_i= (2^6\cdot 3^3+1)\cdot 257 \cdot 65537 \cdot  \prod_{i \geq 2}u_i.\]
Since $2^{16}\mid L$, there is an $n$-pair with $n<6$ or a $6$-prime.
We see from Table \ref{table:P=3} that there is no such a pair and we have two $6$-primes: $193$ and $577$.

\subparagraph{When $q_3=193$}
When $193\mid m$, we have
\[m=257 \cdot 65537 \cdot 7 \cdot 19 \cdot 13 \cdot 193 \prod_{i \geq 3}u_i= (2^8x+1) \prod_{i \geq 3}u_i.\]
As there is no admissible pair, we have the $8$-prime $769$. Combining $769$, we have
\[m=(2^{13}x+1)\prod_{i \geq 4}u_i.\]
As there is no admissible pairs and $13$-primes, no Carmichael number exists in this case.

\subparagraph{When $r_2=577$}
If $577\mid m$, then
\[m=257 \cdot 65537 \cdot 7 \cdot 19 \cdot 13 \cdot577\prod_{i \geq 3}u_i= (2^9x+1) \prod_{i \geq 3}u_i\]
There is no $n$-pair with $n<9$ and there is no $9$ prime.
Hence no Carmichael number exists in this case.

\paragraph{Subsubcase 2: $r_2=37$}
When $37\mid m$, then we have
\[m=257 \cdot 65537 \cdot 7 \cdot 19 \cdot 37 \prod_{i \geq 2}u_i= (2^3\cdot 3 \cdot 5 \cdot 41+1)\cdot 257 \cdot 65537 \cdot  \prod_{i \geq 2}u_i.\]
As $2^{16}\mid L$, $m$ is divisible by the $3$-prime $73$.
Combining this yields
\[m=(2^6x+1)\cdot 257 \cdot 65537 \cdot  \prod_{i \geq 3}u_i.\]
Hence we have a $6$-prime: $193$, $577$.

\subparagraph{When $q_2=193$}
When $193\mid m$, combining this we have
\[m=(2^8x+1)\cdot  \prod_{i \geq 4}u_i.\]
Then we have the $8$-prime $769$, and 
\[m=(2^{12}x+1)\cdot  \prod_{i \geq 5}u_i,\]
which implies again that the $12$-prime $12289$ divides $m$.
Hence
\[m=(2^{13}x+1)\cdot  \prod_{i \geq 6}u_i,\]
but there is no $13$-prime.
Thus no Carmichael number exists in this case.

\subparagraph{When $r_4=577$}
When $577\mid m$, we combine this and obtain
\[m=(2^7x+1)\cdot  \prod_{i \geq 4}u_i.\]
Hence we have the $7$-prime $1153$ and
\[m=(2^8x+1)\cdot  \prod_{i \geq 5}u_i,\]
which implies the existence of the $8$-prime $769$.
Then
\[m=(2^9x+1)\cdot  \prod_{i \geq 5}u_i,\]
but there is no $9$-prime.
Hence no Carmichael number exists in this case.

\subsubsection{Subcase 2: $l_1=s_1=2$}
Consider the case  $l_1=s_1=2$.
As $l_1$ is even, $P=3$ or $P\equiv 1 \pmod{3}$.

\paragraph{Subsubcase 1: When $P=3$}
When $P=3$, we have $q_1=13$ and $r_1=37$.
Note that $2^{16}\mid L$.
Then we have by the minimality argument
\begin{align*}
m&=257\cdot 65537 \cdot 13 \cdot 37 \prod_{i \geq 1}u_i\\
&=(2^5\cdot 3\cdot 5+1) 257\cdot 65537\prod_{i \geq 1}u_i && \text{then the $5$-prime $97$ divides $m$}\\
&=(2^6\cdot 3^6+1)57\cdot 65537\prod_{i \geq 2}u_i.
\end{align*}
This implies that $m$ is divisible by a $6$-prime: $193$, $577$.

\subparagraph{the $6$-prime $193$}
When $193\mid m$, we have
\begin{align*}
m&=(2^{11}x+1) \prod_{i \geq 3}u_i && \text{then the $11$-prime $18433$ divides $m$}\\
&=(2^{12}x+1)\prod_{i \geq 4}u_i&& \text{then the $12$-prime $12289$ divides $m$}\\
&=(2^{13}x+1)\prod_{i \geq 5}u_i.
\end{align*}
(There were no admissible pairs to consider.)
Since there is no $13$-prime, no Carmichael number exists in this case.

\subparagraph{the $6$-prime $577$}
When $577\mid m$, we have
\begin{align*}
m&=257\cdot 65537 \cdot 13 \cdot 37 \cdot 97\cdot 577 \prod_{i \geq 3}u_i\\
&=(2^{7}x+1) \cdot 257\cdot 65537 \prod_{i \geq 4}u_i&& \text{then the $7$-prime $1153$ divides $m$}\\
&=(2^{9}x+1)(2^{16}+1) \prod_{i \geq 4}u_i
\end{align*}
(There were no admissible pairs to consider.)
Since there is no $9$-prime, no Carmichael number exists in this case.

\paragraph{Subsubcase 1: When $P\equiv 1 \pmod{3}$}
Suppose that $P\equiv 1 \pmod{3}$.
Then the only possible $P$ are $7, 13, 241$.
However, as $2^2\cdot 241+1=5\cdot 193$ is not prime, there are two cases: $P=7$ and $P=13$.

\subparagraph{When $P=7$}
When $P=7$, we have $q_1=29$ and $r_1=197$.
Note that $2^{16}\mid L$.
Then the minimality arguments (see Table \ref{table:P=7}) gives 
\begin{align*}
m&=257\cdot 65537 \cdot 29\cdot 197 \prod_{i \geq 1}u_i\\
&=(2^{4}x+1) \cdot 257\cdot 65537 \prod_{i \geq 1}u_i&& \text{then the $4$-prime $113$ divides $m$}\\
&=(2^{6}x+1) \cdot 257\cdot 65537 \prod_{i \geq 2}u_i
\end{align*}
and $m$ is divisible by a $6$-prime: $449$, $3137$.

\begin{enumerate}
	\item When $449\mid m$, we have
	\[m=(2^{7}x+1) \cdot 257\cdot 65537 \prod_{i \geq 3}u_i\]
	but there is no $7$-prime.
	
	\item When $3137\mid m$, we have
	\[m=(2^{9}x+1) \cdot 257\cdot 65537 \prod_{i \geq 3}u_i\]
	but there is no $9$-prime.
\end{enumerate}
Hence no Carmichael number exists in this case.

\subparagraph{When $P=13$}
When $P=13$, then $q_1=53$ and $r_1=677$.
Then we have
\begin{align*}
m&=257\cdot 65537 \cdot 53 \cdot 677 \prod_{i \geq 1}u_i=(2^{3}x+1) \cdot 257\cdot 65537 \prod_{i \geq 1}u_i
\end{align*}
Since $l_2, s_2$ are even, there is no $3$-prime.
Hence no Carmichael number exists in this case.

\subsubsection{Subcase 3: $l_1=s_1=3$}
When $l_1=s_1=3$, we have $P=3$ by Lemma \ref{table:P=3}.
But $2^3\cdot 3+1=5^2$ is not prime.
Hence no Carmichael number exists in this case.

\subsubsection{Subcase 4: $l_1=s_1=4$}
When $l_1=s_1=4$, we have either $P=3$ or $P\equiv 1 \pmod{3}$.
As $2^4\cdot 3+1=7^2$ is not prime, $P\neq 3$.
There are three values for $P\equiv 1 \pmod{3}$ with $k_1=8$: $P=7, 13, 241$.
However, $2^4\cdot 7^2+1=5\cdot 157$, $2^4\cdot 13+1=11\cdot 19$ and $2^4\cdot 241+1=7\cdot 19\cdot 29$ are not prime.
Hence there is no possible $P$ in this case.

\subsubsection{Subcase 4: $l_1=s_1=5$}
Consider the case $l_1=s_1=5$.
As $s_1$ is odd, $P=3$. However $2^5\cdot 3+1=17^2$ is not prime. Hence we have no Carmichael number in this case.

\subsubsection{Subcase 4: $l_1=s_1=6$}
Suppose $l_1=s_1=6$.
As $l_1$ is even, we have either $P=3$ or $P\equiv 1 \pmod{3}$.

\paragraph{When $P=3$}
Suppose that $P=3$. Then $q_1=2^6\cdot 3+1=193$ and $r_1=2^6\cdot 3^2+1=577$.
Recall that $2^{16}\mid L$.
Then we have
\begin{align*}
m&=257\cdot 65537 \cdot 193 \cdot 577 \prod_{i \geq 1}u_i=(2^{10}x+1) \cdot (2^{16}+1) \prod_{i \geq 1}u_i
\end{align*}
There is no $n$-pair with $6<n<10$ and there is no $10$-prime (See Table \ref{table:P=3}).
Hence no Carmichael number exists in this case by the minimality argument.

\paragraph{When $P\equiv 1 \pmod{3}$}
Consider the case $P\equiv 1 \pmod{3}$.
Then $P=7, 13, 241$. However $2^6\cdot 13+1=7^2\cdot 17$ and $2^6 \cdot 241+1=5^2\cdot 617$ are not prime.

When $P=7$, we have $q_1=449$ and $r_1=3137$.
Then
\begin{align*}
m&=257\cdot 65537 \cdot 449 \cdot 3137 \prod_{i \geq 1}u_i=(2^{8}x+1) \prod_{i \geq 1}u_i
\end{align*}
As there is no $n$-pair with $6<n<8$ and $8$-prime (see Table \ref{table:P=7}) , there is no Carmichael number in this case.

\subsubsection{Subcase 4: $l_1=s_1=7$}
If $l_1=s_1=7$, then $P=3$ as $s_1$ is odd.
However, $2^7\cdot 3 +1=5 \cdot 7 \cdot 11$ is not prime.
Thus, no Carmichael number exists in this case by the minimality argument.

This exhausted all the cases, and this completes the proof of Theorem \ref{thm:main theorem}.


\end{document}